\documentclass[11pt]{article}
\usepackage{amsbsy}
\usepackage{amssymb,amsthm, amsmath, latexsym}
\usepackage{mathrsfs}
\usepackage{physics} 
\usepackage{esint} 
\usepackage{braket} 
\usepackage{color}
\usepackage[colorlinks=true,citecolor={Plum},linkcolor={red}]{hyperref}
\usepackage{pdfsync}
\usepackage{footmisc}
\usepackage{dsfont}
\usepackage{graphicx,epsfig,subfigure,psfrag}
\usepackage{color}
\usepackage{enumerate}
 \usepackage[usenames,dvipsnames]{pstricks}

\newtheorem{theorem}{Theorem}[section]
\newtheorem{proposition}{Proposition}[section]

\newtheorem{lemma}{Lemma}[section]
\newtheorem{corollary}{Corollary}[section]
\newtheorem{definition}{Definition}[section]
\newtheorem{remark}{Remark}[section]

\numberwithin{equation}{section} \numberwithin{theorem}{section}
\numberwithin{proposition}{section} \numberwithin{lemma}{section}
\numberwithin{corollary}{section}
\numberwithin{definition}{section} \numberwithin{remark}{section}

\newcommand{\R}{\mathbb{R}}
\newcommand{\N}{\mathbb{N}}

\date{\today}

\title{Minimization of the first eigenvalue for the Lam\'e system}
\author{Antoine Henrot\footnote{Universit\'e de Lorraine, CNRS, Institut Elie Cartan de Lorraine, BP 70239 54506 Vand\oe uvre-l\`es-Nancy Cedex, France ({\tt antoine.henrot@univ-lorraine.fr}).}
\and Antoine Lemenant\footnote{Universit\'e de Lorraine, CNRS, Institut Elie Cartan de Lorraine, BP 70239 54506 Vand\oe uvre-l\`es-Nancy Cedex, France ({\tt antoine.lemenant@univ-lorraine.fr}).}~\footnote{Institut Universitaire de France (IUF)}
\and Yannick Privat\footnote{Universit\'e de Lorraine, CNRS, Institut Elie Cartan de Lorraine, Inria, BP 70239 54506
Vandœuvre-l\`es-Nancy Cedex, France. ({\tt yannick.privat@univ-lorraine.fr}).}~\footnotemark[3]
}

\date{\today}

\begin{document}
\maketitle

\begin{abstract}
In this article, we address the problem of determining a domain in $\R^N$ that minimizes the first eigenvalue of the Lamé system under a volume constraint. We begin by establishing the existence of such an optimal domain within the class of quasi-open sets, showing that in the physically relevant dimensions $N = 2$ and $3$, the optimal domain is indeed an open set. Additionally, we derive both first and second-order optimality conditions. Leveraging these conditions, we demonstrate that in two dimensions, the disk cannot be the optimal shape when the Poisson ratio is below a specific threshold, whereas above this value, it serves as a local minimizer. We also extend our analysis to show that the disk is nonoptimal for Poisson ratios $\nu$ satisfying $\nu \leq 0.4$.
\end{abstract}

 \tableofcontents

\section{Introduction}
The Faber–Krahn inequality is one of the most fundamental results in spectral geometry. It states that, among all sets of a given volume (in any dimension), the ball uniquely minimizes the first eigenvalue of the Dirichlet–Laplacian, see \cite{Faber1923}, \cite{Krahn}.
Similar results hold for other boundary conditions. For instance:
\begin{itemize} 
\item the ball maximizes the first non-trivial eigenvalue of the Neumann–Laplacian (Szeg\"o-Weinberger see \cite{Szego}, \cite{Weinberger}),
\item the ball minimizes the first eigenvalue of the Robin-Laplacian (Bossel-Daners (when the boundary parameter is positive) see \cite{Bossel}, \cite{Daners}), 
\item the ball maximizes the first (non-trivial) eigenvalue 
of the Steklov-Laplacian (Brock see\cite{Brock}). 
\end{itemize}
In all of these problems, a volume constraint is imposed. For further discussion on eigenvalue optimization problems, see \cite{henrot06}, \cite{henrot18}.

The question of minimizing or maximizing the first eigenvalue in the case of systems is far less understood. Notably, unlike the scalar cases discussed earlier, the ball does not necessarily serve as the extremal domain. Recent studies have explored this question for the Stokes operator.
It is shown that in three dimensions, the ball does not minimize the first eigenvalue among sets of a given volume. In two dimensions, however, the disk is found to be a local minimizer and is conjectured to be the global minimizer. A numerical investigation of this problem is presented in \cite{li-zhu23}. Another operator that has recently attracted attention in a series of studies is the curl operator, see \cite{Cantarella_2000_Bis,Cantarella_2000},\cite{Gerner23}, 
\cite{enciso-peralta}, \cite{enciso2022optimal}. These studies also examine various properties of potential optimal domains, revealing that the ball does not minimize the first eigenvalue. In \cite{enciso-peralta}, the authors demonstrate, under certain regularity conditions, that an optimal domain cannot possess axial symmetry. Conversely, it is conjectured in \cite{Cantarella_2000_Bis} that the optimal domain is a {\it spheromak}, namely a torus in $\mathbb{R}^3$ with its central hole minimized to form an almost spherical shape, often likened to a cored apple.

In the recent paper \cite{KLZ24}, the authors investigate the Maxwell operator (or vectorial Laplacian) with the boundary condition $ u \times \nu = 0$. They demonstrate that, in three dimensions, the ball is neither a minimizer nor a maximizer for the first eigenvalue under volume or perimeter constraints. Specifically, the authors show that the infimum of the first eigenvalue is zero, while the supremum is $ +\infty $ under both constraints.

In this article, we focus on the Lamé system with Dirichlet boundary conditions, which is fundamental to the theory of linear elasticity. Let $\Omega$ be a bounded open set in $\R^N$ and $H^1_0(\Omega)^N$ denotes the space of vectors $u=(u_1, \ldots, u_N)$
where each $u_i$ belongs to the Sobolev space $H^1_0(\Omega)$.
The first eigenvalue of $\Omega$ for the Lam\'e system is defined by
 \begin{equation}\label{defLambda}
\Lambda(\Omega):=\min_{u\in H^1_0(\Omega)^N \setminus \{0\}} \frac{\mu\int_{\Omega}|\nabla u|^2 \;dx + (\lambda + \mu) \int_{\Omega} ({\rm div}(u))^2\;dx}{\int_{\Omega} |u|^2 \;dx},
 \end{equation}
 where $\lambda , \mu$ are the Lam\'e coefficients that satisfy $\mu>0,\lambda+\mu >0$. In the previous expression, $|\nabla u|^2$ denotes
 $|\nabla u_1|^2+\ldots |\nabla u_N|^2$ and $|u|^2$ denotes $u_1^2+\ldots u_N^2$.
 The associated PDE solved by the minimizer $u$ is
 \begin{equation}\label{pdeLame0}
 \left\lbrace
 \begin{array}{cc}
 -\mu \Delta u -(\lambda+\mu) \nabla ({\rm div}(u)) = \Lambda u & \mbox{ in } \Omega,\\
 u=0  & \mbox{ on } \partial\Omega .\\
 \end{array}
 \right.
 \end{equation}
 We will explain below that a natural motivation for introducing the first eigenvalue $\Lambda$ arises from the famous Korn inequality. 
In that context, let us mention the paper \cite{Lewicka-Muller} where optimal constants for the Korn inequality are also computed
but under tangential boundary conditions.

 It is worth noting that it will be convenient to introduce the Poisson coefficient $\nu$, as it is related to the Lam\'e parameters through the following formulae:
 \begin{equation}\label{lamepoi}
 \lambda=\frac{E\nu}{(1+\nu)(1-2\nu)},\qquad \mu=\frac{E}{2(1+\nu)}
 \end{equation}
 where $E$ is the Young modulus and $\nu \in (-1,0.5)$ (for many materials $\nu \in [0.2,0.4]$).
 Indeed, dividing the eigenvalue $\Lambda$ by $\mu$ lead us to introduce the ratio 
 $$\frac{\lambda+\mu}{\mu}= \frac{1}{1-2\nu}$$
 and therefore, we can see that the minimization of $\Lambda$ will primarily depend on the Poisson coefficient $\nu$.
 In some papers, like \cite{kawohl-sweers}, the authors look at an eigenvalue defined as
  \begin{equation}\label{defLambda2a}
\Lambda(\Omega,a):=\min_{u\in H^1_0(\Omega)^N \setminus \{0\}} \frac{\int_{\Omega}|\nabla u|^2 \;dx + a \int_{\Omega} ({\rm div}(u))^2\;dx}{\int_{\Omega} |u|^2 \;dx},
 \end{equation}
 where $a$ stands for $1/(1-2\nu)$. To make the underlying physics more apparent, we will explicitly retain the Lam\'e parameters and the Poisson coefficient in our paper.

Thus, this paper is dedicated to the study of the following shape optimization problem: 
\begin{equation}\label{shop1}
\boxed{ \inf\{ \Lambda(\Omega), \Omega\subset \R^N \mbox{ bounded }, |\Omega|=V_0\}}
 \end{equation}
 or equivalently (since $\Lambda(t\Omega)=\Lambda(\Omega)/t^2$) to the unconstrained optimization problem
  \begin{equation}\label{shop2}
 \inf\{|\Omega|^{2/N} \Lambda(\Omega), \Omega\subset \R^N \}.
 \end{equation}
 Here $|\Omega|$ denotes the Lebesgue measure of the open set (or quasi-open set) $\Omega$. The precise definition of quasi-open set
 will be given at the beginning of Section  \ref{secexistence}.

In Section \ref{secexistence}, we first establish an existence result for quasi-open sets in any dimension. To achieve this, we employ the standard strategy based on a concentration-compactness argument, that we had to adapt in a vectorial context. Subsequently, we demonstrate some mild regularity, specifically that the optimal domain is an open set in the physical dimensions 2 and 3. This is accomplished by proving the equivalence of the minimization problem with a penalized problem and introducing the concept of Lamé quasi-minimizers. We then show that these quasi-minimizers are globally H\"older continuous.
Therefore our first important result is
\begin{theorem}\label{theoexis}
There exists a quasi-open set $\Omega^*$ solution of \eqref{shop1} or \eqref{shop2}. Moreover in dimension $N=2$ and $N=3$ this set is open and  any eigenfunction associated with   the eigenvalue $\Lambda(\Omega^*)$   belongs to $\mathscr{C}^{0,\alpha}(\R^N)$ for all $\alpha<1$ if $N=2$, and for all $\alpha<\frac{1}{2}$ if $N=3$.
\end{theorem}
In Section \ref{secoptim}, we derive first and second order optimality conditions by calculating the first and second shape derivatives of the eigenvalue. These computations prove to be particularly useful in the subsequent Section \ref{secdisk}, where we examine the potential optimality of the disk in two dimensions. In this context, we are able to prove:
\begin{theorem}
If the Poisson coefficient $\nu$ is less than $0.4$, the disk {\bf is not} the minimizer of $\Lambda$ (among sets of given area).
\end{theorem}
The proof involves several steps. First, we explicitly compute the first eigenvalue of the disk, which is a non-trivial task, and show that if $\nu \leq 0.349\ldots$, the first eigenvalue is double. This finding allows us, through a straightforward variational argument, to conclude that the disk cannot be optimal in this case. It is worth noting that the value $0.349\ldots$ is explicitly related to the first zero of the Bessel function $J_1$ and its derivative.

Next, in Section \ref{secrhombi}, we identify explicit rhombi that yield a better first eigenvalue than the disk for the range $0.349\ldots \leq \nu \leq 0.3878\ldots$. We further extend our analysis by considering suitable rectangles in Section~\ref{secrectangle}.

For these rectangles, we cannot compute the first eigenvalue explicitly, but we can provide precise estimates through a clever choice of test functions. This approach allows us to rule out the disk as a possible minimizer for $\nu \leq 0.4$. For values of $\nu$ between $0.4$ and $0.5$, we cannot reach a conclusion using analytical arguments. However, in Section \ref{seconddisk}, we demonstrate that once the first eigenvalue is simple (i.e., for $\nu > 0.349\ldots$), the disk serves as a local minimizer for our problem. This is established by showing that the second shape derivative is non-negative, and we can also estimate the quadratic form using the $H^1$-norm of the perturbation.

Finally, we present heuristic arguments suggesting that there exists a threshold $\nu^*$ such that the disk could be a minimizer when $\nu^* \leq \nu < 0.5$. This conclusion is based on the property that the Lamé eigenvalue $\Gamma$-converges to the Stokes eigenvalue as $\nu \to 1/2$, in conjunction with the previously established local minimality of the disk.

\section{Motivation and elementary comparisons}
  \subsection{Reminders on the Korn inequalities}
  Let $\Omega$ denote any bounded open set in $\R^N$. 
For $u:\Omega \subset \R^N \to \R^N$ we denote by $e(u)$ the symmetric gradient defined by 
 $$e(u):=\frac{\nabla u+\nabla u^T}{2}.$$
 
Let us recall two standard Korn inequalities.
  
  \begin{theorem}\label{theokorn} 
For all  $u \in H^1_0(\Omega)^N$, one has
  \begin{equation} \tag{\text{Korn}}  \|\nabla u\|_{L^2(\Omega))}\leq 2 \| e(u)\|_{L^2(\Omega)}, \label{korn1}
  \end{equation}
  and
  \begin{equation}\tag{\text{Poincar\'e-Korn}}
\|u\|_{L^2(\Omega))}\leq C(\Omega) \| e(u)\|_{L^2(\Omega)}. \label{korn2}
  \end{equation}
      Moreover we can take $C(\Omega)= 2/\lambda_1^D(\Omega)$, where $\lambda_1^D(\Omega)$ is the usual scalar first eigenvalue for the Dirichlet Laplacian.
  \end{theorem}
 
 \begin{proof}
 An elementary  ``integration by parts'' reveals  that the following identity is always true for any  $u \in C^\infty_c(\Omega, \R^N)$,
 $$\int_{\Omega}|e(u)|^2 dx=\frac{1}{2}\left(\int_{\Omega} |\nabla u|^2 + ({\rm div}(u))^2\right).$$
Inequality \eqref{korn1} follows immediately. Notice that here the constant does not depend on $\Omega$. 
 
Now for a proof  of \eqref{korn2}, we apply \eqref{korn1} and the following Poincaré inequality,   for a scalar function $v \in H^1_0(\Omega))$:
 
 $$\lambda_1(\Omega)\int_{\Omega} v^2  \; dx \leq   \int_{\Omega}  |\nabla v|^2 \; dx.$$
 We deduce that, for $u=(u_1,u_2,\ldots u_N)$:
 
 \begin{eqnarray}
 \frac{\int_{\Omega} |e(u)|^2 \;dx }{\int_{\Omega} u^2 \;dx}&=&\frac{\frac{1}{2}\left(\int_{\Omega} |\nabla u|^2 + ({\rm div}(u))^2\right)}{\int_{\Omega}\sum_{i=1}^N (u_i)^2 }  \geq \frac{\frac{1}{2}\left(\int_{\Omega} \sum_{i=1}^N |\nabla u_i|^2 \right)}{\int_{\Omega} \sum_{i=1}^N (u_i)^2}  \notag \\
&\geq &  \min_{i=1,2,\ldots N} \left( \frac{\frac{1}{2}\int_{\Omega} |\nabla u_i|^2}{\int_{\Omega}(u_i)^2}  \right) \geq \frac{1}{2}\lambda_1(\Omega). \notag 
 \end{eqnarray}  
 
\end{proof}

 Therefore, looking at the best constant in the \eqref{korn2} leads us to compute the eigenvalue $\Lambda$ defined in \eqref{defLambda} for the particular choice $\mu=1/2$ and $\lambda=0$.

 \subsection{Link with other eigenvalues}
\subsubsection{Link with the eigenvalues of the Stokes operator}\label{remark1} 
For $\Omega$, a bounded open set, let us introduce the so-called {\it Dirichlet Stokes first eigenvalue} $\lambda_1^{\rm Stokes}(\Omega)$ by
$$
\lambda_1^{\rm Stokes}(\Omega):= \inf_{\substack{u \in (H^1_0(\Omega))^N\setminus \{0\} \\ {\rm div}(u)=0\text{ in }\Omega}} 
\frac{\int_{\Omega} |\nabla u|^2 } {\int_{\Omega} |u|^2 }.
$$
Then for all $\Omega$ it holds
\begin{equation}\label{comp:StKorn}
\mu \lambda_1^{\rm Stokes}(\Omega) \geq  \Lambda(\Omega).
\end{equation}
Indeed this inequality comes from the fact that the space on which we minimize is a subspace of $H^1_0(\Omega)^N$ on which the energy coincides with our Rayleigh quotient.

In some sense, the divergence term may appear as a penalization term, in particular when the Poisson coefficient goes to $1/2$
(or equivalently when the Lam\'e coefficients are such that $(\lambda + \mu)/\mu \to +\infty$). We will make this more precise in Section
\ref{secconclusion} by proving the strong convergence of the Lam\'e operator to the Stokes operator when $\nu \to 1/2$.
For that purpose, we will use the tool of $\Gamma$-convergence.


 \subsubsection{Comparison with Dirichlet Eigenvalues}
 
 In this section we retrieve some results that already appeared, for example in \cite{kawohl-sweers}. We recall that 
 $\lambda_1^D(\Omega)$ denotes the first eigenvalue of the Dirichlet-Laplacian.
 
\begin{proposition} \label{boundlambda1} For any bounded domain $\Omega$ it holds
\begin{eqnarray}
\mu \lambda_1^D(\Omega) <\Lambda(\Omega) \leq \frac{(\lambda + (N+1)\mu)}{N} \, \lambda_1^D(\Omega).\label{ineqlambda1}
\end{eqnarray}
Moreover, 
$$\inf_{\Omega} \frac{\Lambda(\Omega)}{\lambda_1^D(\Omega)}=\mu  ,$$
and is achieved by a sequence of thin cuboids shrinking to a line.
\end{proposition} 
\begin{remark}
From the left inequality in \eqref{ineqlambda1} and the famous Faber-Krahn inequality, we observe that
$$|\Omega|^{2/N} \Lambda(\Omega) \geq \mu \, |\Omega|^{2/N} \lambda_1^D(\Omega) \geq \mu j_{N/2-1,1}^2$$
where $j_{N/2-1,1}$ is the first zero of the Bessel function $J_{N/2-1}$. Thus we see that the infimum in \eqref{shop2} is strictly
positive. 
\end{remark}
\begin{proof} The inequality
$$\mu \lambda_1^D(\Omega) \leq \Lambda(\Omega)$$
follows the chain of inequalities that appear in the proof of Theorem \ref{theokorn} (multiplied by $\mu$ instead of $1/2$).

Now we demonstrate that the inequality must be strict. Indeed, assuming that $\mu \lambda_1^D(\Omega) = \Lambda(\Omega)$ and applying the previously established chain of inequalities, we observe that each $u_i$ must be a Dirichlet eigenfunction associated with $\lambda_1^D(\Omega)$.
But since we  also have ${\rm div}(u)=0$ we will get a contradiction according to   Lemma \ref{div0} below, and this achieves the proof   of the strict inequality.

We now prove the upper bound. For that purpose we consider $u_1$ being the (normalized) Dirichlet eigenfunction associated to $\lambda_1(\Omega)$ and we consider the vector test functions  $(0,0,\ldots,u_1,0,\ldots)$ composed of null functions except
$u_1$ in $i$-th position. Then 
$$\Lambda\leq \frac{\mu \int_{\Omega}|\nabla u_1|^2\;dx + (\lambda +\mu)\int_{\Omega} (\partial_i u_1)^2 \;dx}{\int_{\Omega} u_1^2 \;dx}.$$
Summing these $N$ inequalities yields
$$N \Lambda \leq N\mu \int_{\Omega}|\nabla u_1|^2\;dx +(\lambda +\mu)  \int_{\Omega}|\nabla u_1|^2\;dx =(\lambda + (N+1)\mu)\lambda_1^{D},$$
which finishes the proof of \eqref{ineqlambda1}.

Let us now prove the last assertion. For that purpose, we consider the cuboid $\Omega_L=(0,L)\times \prod_{i=2}^N (0,1)$
and take a first Dirichlet eigenfunction of $\Omega_L$ namely
$$u_1(X)=\sin\left(\pi \frac{x_1}{L}\right) \prod_{i=2}^N \sin(\pi x_i).$$
We will use the fact that
$$\lambda_1^D(\Omega)=\pi^2 (N-1+\frac{1}{L^2}).$$
Now, we plug in the Rayleigh quotient defining $\Lambda(\Omega_L)$ the vector $u=(u_1,0,\ldots,0)$.
Since
$$\int_{\Omega_L} |u|^2 \;dx=\frac{L}{2^N}$$
$$\int_{\Omega_L}|\nabla u|^2 \; dx= \int_{\Omega_L}|\nabla u_1|^2 \; dx=\frac{L \pi^2}{2^N}(N-1+\frac{1}{L^2})=
\frac{L}{2^N} \lambda_1^D(\Omega_L$$
$$\int_{\Omega_L}({\rm div}(u))^2 \; dx =\frac{ \pi^2}{L^2}\frac{L}{2^N}$$
we deduce that
\begin{eqnarray}
\Lambda (\Omega_L)\leq  \frac{\mu \frac{L}{2^N} \lambda_1^D(\Omega_L) + (\lambda+\mu) \frac{L}{2^N} \frac{\pi^2}{L^2}}{\frac{L}{2^N}} \notag 
\end{eqnarray}
or
$$\Lambda(\Omega_L) \leq \mu \lambda_1^D(\Omega_L) +  (\lambda+\mu) \frac{L}{2^N} \frac{\pi^2}{L^2}$$
and finally letting $L\to +\infty$ we conclude that
$$\inf_{\Omega}  \frac{\Lambda(\Omega)}{\lambda_1^D(\Omega)}=\mu,$$
as claimed in the proposition.
\end{proof}


\begin{lemma}\label{div0}  Let $u=(u_1,u_2,\ldots, u_N)$ be a $N$-uple of functions in $H^1_0(\Omega)$ such that ${\rm div}(u)=0$ and for all $i$, $u_i=\alpha_i u_1$ for some $\alpha_i \in \R$. Then $u_1$ and then all the $u_i$ are identically zero.\\
In particular, if $\Omega$ is connected and all the $u_i$ are eigenfunctions associated to the first eigenvalue $\lambda_1^D(\Omega)$,
it is not possible that ${\rm div}(u)=0$.
\end{lemma}

\begin{proof} From the assumptions  ${\rm div}(u)=0$ and $u_i=\alpha_i u_1$  we deduce that $u_1$ satisfies 
$$\frac{\partial u_1}{\partial x_1} + \sum_{i=2}^N \alpha_i \frac{\partial u_1}{\partial x_i} =0$$ 
which means that $u_1$ must be constant on all affine lines directed by $(1,\alpha_2,\ldots,\alpha_N)$. Since all those lines touches the boundary of $\Omega$,  from  the Dirichlet condition on $u_1$ we deduce that $u_1$ must be identically $0$.\\
The last assertion comes from the fact that the first Dirichlet eigenvalue (for the Laplacian) of a connected domain is simple.
\end{proof}



 \section{Existence and regularity }\label{secexistence}
\subsection{Existence of an optimal  quasi-open  set}

In this section, we will fix the values of the Lamé coefficients, specifically choosing $\mu = 1/2$ and $\lambda = 0$, which corresponds to the Korn inequality. This choice does not affect the proof of existence (since the general case would simply involve multiplying the terms by positive constants), but it simplifies the proof and enhances its clarity.

We establish the existence of an optimal shape within the class of quasi-open sets, employing a standard concentration-compactness strategy first introduced by Lions \cite{lions}. This approach has been utilized to solve shape optimization problems for the Laplace operator, initially by Dorin Bucur (see \cite{dorin1,dorin2,dorin3,depveli}), and subsequently by several other authors. Recently, this strategy has also been applied to the Stokes operator \cite{henrot-mazari-privat}. We denote by $\operatorname{Cap}(A)$ the $H^1$-capacity of $A$ (for instance   the Bessel capacity). A set $A\subset \R^N$ is said to be quasi-open if, for every $\varepsilon>0$ there exists an open set $\Omega_\varepsilon$ such that $A\subset \Omega_{\varepsilon}$ and $\operatorname{Cap}(\Omega_{\varepsilon}\setminus A)\leq \varepsilon$. We first introduce the class
$$\mathcal{O}:=\{ \Omega \subset \R^N \text{ quasi-open  such that } 0<|\Omega|< +\infty \}.$$
The space $H^1_0(\Omega)$ is defined as functions $u \in H^1(\R^N)$ such that $u=0$ quasi-everywhere on $\Omega^c$.  
Notice that a domain $\Omega \in \mathcal{O}$ is not necessarily bounded.  However, the space $H^1_0(\Omega)$ is known to be a closed subspace of $H^1(\R^N)$ which is compactly embedded  into $L^2(\R^N)$, when $|\Omega|<\infty$ (because by definition of being quasi-open there exists an open set $E$ with $|E|<+\infty$ such that $\Omega \subset E$ thus $H^1_0(\Omega)\subset H^1_0(E)$ and the standard compact embedding of $H^1_0(E)$ into $L^2$ applies).

Notice also that thanks to Proposition \ref{remarkK} below, the space of all  $u \in L^2(\Omega)^N$ such that $e(u)\in L^2(\Omega)$ and $u=0$ $\operatorname{Cap}_{1,2}$-q.e. in $\Omega^c$ coincides with the space $H^1_0(\Omega)^N$.

Then we can relax the definition of $\Lambda(\Omega)$ for $\Omega \in \mathcal{O}$  by considering
$$\Lambda(\Omega):= \min_{u \in H^1_0(\Omega)^N} \frac{\int_{\R^N} |e(u)|^2 \; dx}{\int_{\R^N}|u|^2 \;dx}.$$
Notice here  that $\Omega$ is merely quasi-open and not necessarily open, but the definition coincides with the standard one when $\Omega$ is open.  

  Also, it is easy to check that the minimum in the definition of $\Lambda$ is achieved by an $H^1_0(\Omega)$ function, thanks to the compact embedding of $H^1_0(\Omega)$   into $L^2(\R^N)$ and the semicontinuity behavior of the convex functional $u\mapsto \int_{\R^N} |e(u)|^2 \; dx$ for the weak topology of $H^1$.

In the sequel we will need the famous Korn inequality but now in the whole $\R^N$, in particular valid   for  $u\in H^1_0(\Omega)^N$ with $\Omega \in \mathcal{O}$ quasi-open, as stated in the following proposition.

\begin{proposition}[Korn inequality in $\R^N$] \label{remarkK}If   $u \in L^2(\R^N)^N$ is  such that $e(u)\in L^2(\R^N)$, then $u\in H^1(\R^N)^N$ and 
\begin{eqnarray}
 2\int_{\R^N } |e(u)|^2 \;dx =  \int_{\R^N  } |\nabla u|^2 +({\rm div} (u))^2 \; dx.  \label{kornkornkorn}
\end{eqnarray}
\end{proposition}
\begin{proof}  Let $R>0$ be given and let $\varphi_{ R} \in [0,1]$ be a cut-off function such that $\varphi_{\lambda, R}=1$ on $B(0,R)$,  $\varphi_{\lambda,R}=0$ on $B(0,2 R)^c$, and 
$$|\nabla \varphi_{\lambda,R}|\leq C\frac{1}{R}.$$
Then the function $u_R:=\varphi_{R} u$ clearly belongs to $H^1_0(B(0,2 R))^N$ and  Korn's inequality   \eqref{kornkornkorn}  holds true for the function $\varphi_{R} u$.  For simplicity we will by now denote by $\varphi$ the function $\varphi_{R}$. Notice that 
$$
e(\varphi u)_{ij}=\frac{u^i\partial_j\varphi +u^j\partial_i\varphi }{2}+\varphi e(u)_{i,j},
$$
so  pointwisely in $\mathbb{R}^N$ it holds the following estimate
$$|e(u  \varphi )| \leq  |u| |\nabla \varphi | +  |\varphi | |e(u)|\leq \frac{C}{ R} |u| + |e(u)|,$$
$$|{\rm div}(u \varphi)|\leq \frac{C}{ R} |u| + |{\rm div}( u)|,$$

$$|D (u\varphi)|\leq \frac{C}{R} |u| + |D(u)|.$$

Recall also that  $u\varphi = u$ in $B(0,R)$. Now  applying \eqref{kornkornkorn} in $B(0, R)$ to the function $\varphi u$ we obtain

\begin{eqnarray}
2\int_{\R^N\cap B(0,R)} |e(u)|^2 \;dx =  \int_{\R^N \cap B(0,R)} |\nabla u|^2 +({\rm div} (u))^2 \; dx +E(R), \label{aap}
\end{eqnarray}
with 
$$|E(R)| \leq \frac{C}{R^2} \int_{\Omega \setminus  B(0,2 R)} |u|^2 \;dx.$$
We now  let $R\to +\infty$ which yields, 
$$|E(R)|  \underset{R\to +\infty}{\longrightarrow}  0.$$
Thus passing \eqref{aap} to the limit, and using that $e(u) \in L^2(\R^N)$ we can use Fatou lemma to get first $\nabla u \in L^2(\R^N)$,  after which  the monotone convergence theorem allows to conclude
\begin{eqnarray}
 2\int_{\R^N } |e(u)|^2 \;dx =  \int_{\R^N  } |\nabla u|^2 +({\rm div} (u))^2 \; dx.  \notag
\end{eqnarray}
This proves that $u \in H^1(\R^N)^N$ and finishes the proof.
\end{proof}

The purpose of this section is to prove the following result.

\begin{theorem}{\label{existence} } For all $V>0$ there exists a solution for the problem
$$\min_{\Omega \in \mathcal{O} \; \text{ such that }\; |\Omega| \leq V} \Lambda(\Omega).$$
\end{theorem}

\begin{proof} The proof follows the same approach as in \cite{henrot-mazari-privat} reasoning on the scalar function $|u|$ and using the concentration-compactness strategy of Lions \cite{lions}.  More precisely, we let $\Omega_k$ be a minimizing  sequence with $|\Omega_k|\leq V$ and we consider $w_k:=|u_k|$ where $u_k$ is a chosen normalized eigenvalue for  $\Lambda(\Omega_k)$. In other words, $\|w_k\|_{L^2(\R^N)}=1$ and by  Proposition~\ref{remarkK},
\begin{eqnarray}
\int_{\R^N} |\nabla (w_k)|^2 \;dx  \leq \int_{\R^N} |\nabla u_k|^2 \; dx \leq 2\int_{\R^N} |e(u_k)|^2 \;dx = 2\Lambda(\Omega_k)\leq C_0, \label{boundD}
\end{eqnarray}
so that $w_k$ is uniformly bounded in $H^1(\R^N)$. Let  $Q_n:\R^+\to \R^+$ be the sequence of concentration functions\footnote{According to Lions \cite{lions} this notion was first introduced by L\'evy.}  defined by 
$$Q_n(R):= \sup_{y\in \R^N}\int_{B(y,R)}|u_k|^2 \;dx.$$
Then $Q_n$ is a sequence of nondecreasing functions on $\R^n$ which are uniformly bounded by $1$. By Dini's theorem, up to extract a subsequence (not relabelled), $(Q_n)_n$  admits a pointwize limit function which is nondecreasing, bounded by 1, and that we denote  $Q:\R^+ \to \R^+$. Then we let 
$$\alpha:=\lim_{R\to +\infty} Q(R) \in [0,1].$$
The value of $\alpha$ is usually referred to the ``maximal concentration''. Depending on the value of $\alpha$, we know that one of the following occurs by the concentration-compactness principle of Lions \cite[Lemma I.1]{lions}.

\begin{itemize}
\item {\bf If $\alpha=1$: Compactness:} There exists   a sequence $(y_k)_{k \in \mathbb{N}}$ such that $|w_k|^2(\cdot-y_k)$ is tight:
$$\forall \varepsilon >0, \exists R< +\infty , \forall k \quad \int_{y_k+B_R} w_k^2 \;dx\geq 1- \varepsilon.$$
\item {\bf If $\alpha \in (0,1)$: Dichotomy:} There exist  $(y_k)_{k \in \mathbb{N}}$ and two sequences of positive radii $(R_k)_{k}$, $(R_k')_k$ satisfying
$$R'_k - R_k \to +\infty \text{ and } R_k,R_k'\to +\infty,$$
and such that 
\begin{eqnarray}
\int_{B(y_k,R_k)} w_k^2 \to \alpha, \text{ } \int_{B(y_k,R_k')^c} w_k^2 \to 1-\alpha. \label{information1}
\end{eqnarray}
\item {\bf If $\alpha=0$: Vanishing.} For every $R>0$, 
$$\lim_{k\to +\infty} \sup_{y \in \R^N} \int_{B(y,R)} w_k^2 =0.$$
\end{itemize}

As usual, our aim is to prove that only the compactness case can occur, by ruling out the two other cases.   Let us first prove that the compactness situation implies the desired existence.
\medskip

\noindent{\bf Step 1.} \emph{Compactness implies existence.}  We consider the sequence of translated functions  $u_k(\cdot -y_k)$ that we still denote by $u_k$. We know by assumption that this sequence is uniformly bounded in $H^1(\R^N)$ thus admits a weakly converging subsequence.  Since $H^1(\R^N)$ is compactly embedded in $L^2_{loc}(\R^N)$, using a diagonal argument we can extract a subsequence (not relabelled)  and a function  $u \in L^2_{loc}(\R^N)$ such that $u_k\to u$ strongly in $L^2_{loc}$ and weakly in $H^1(\R^N)$.  Now we use that  $(w_k)$ is in the situation of compactness, and in particular for every $\varepsilon >0$ there exists  $R>0$ such that 
$$\int_{B_R} |u_k|^2 \;dx\geq 1 - \varepsilon.$$
Passing to the limit and using the convergence of $u_k$ in  $L^2(B_R)$ we deduce that $\int_{B_R} |u|^2\; dx \geq  \; 1-\varepsilon$, which  means in particular that 
$$\int_{\R^N} |u|^2\; dx \geq  \; 1-\varepsilon,$$
and since $\varepsilon$ is arbitrary, we finally get $\int_{\R^N} |u|^2\; dx \geq  \; 1$. But of course the reverse is also true so in conclusion $\|u\|_{L^2(\R^N)}=1$.  But we already knew that $u_k$ was converging weakly in $L^2(\R^N)$ to $u$. We just have proved that the sequence of norms are also converging so finally $u_k$ converges strongly in $L^2(\R^N)$ to $u$. Passing to the limit in the Rayleigh quotient, strongly in $L^2$ for $u_k$ and weakly  in $L^2$ for $e(u_k)$ we deduce that 
\begin{eqnarray}
\inf_{\Omega \in \mathcal{O} \; \text{ such that }\; |\Omega| \leq V} \Lambda(\Omega)= \frac{\int_{\R^N} |e(u)|^2 \;dx}{\int_{\R^N} |u|^2 \;dx}. \label{infimum}
\end{eqnarray}
Let us denote $\Omega=\{|u|>0\}$, which is a quasi-open set, and from the equality in \eqref{infimum} we know that $u$ must be an eigenfunction associated to $\Lambda(\Omega)$. Furthermore, since $\int_{\R^N} |u|^2 \;dx=1$ we know that $|\Omega|>0$.  We can also assume that $|u_k|$ converges a.e. in $\R^N$ to $|u|$. This implies, for a.e. $x\in \R^N$,
$$\mathds{1}_{\{|u|>0\}}(x) \leq  \liminf_{k}\mathds{1}_{\{|u_k|>0\}}(x),$$
 and since $|\{|u_k|>0\}|\leq V$, we deduce by Fatou Lemma that  $|\Omega|\leq V$ and finally $\Omega$ is a solution.

\medskip

\noindent{\bf Step 2.} \emph{Vanishing does not occur.} This case is easy to exclude by standard arguments. Indeed, Lemma 3.3 in \cite{zbMATH01482128} says   that up to a subsequence, $w_k(\cdot +y_k)$ does not weakly converge in $H^1(\mathbb{R}^N)$, which is a contradiction with the uniform bound in  \eqref{boundD} together with the fact that $\|w_k\|_2=1$, implying  that $w_k$ is uniformly bounded in $H^1(\mathbb{R}^N)$ thus admits a weakly converging subsequence. 

\medskip

\noindent{\bf Step 3.} \emph{Dichotomy cannot occur.} Assume that $(w_k)_{k\in \mathbb{N}}$ is in the dichotomy situation. Then the idea is to split the minimizing sequence in two disjoint pieces. For that purpose we define $\eta_k:=(R_k'-R_k)/4$ and then we construct two cut-off functions: the first one $\varphi_{k,1}$ supported in $B(y_k,R_k+2\eta_k)$ is such that $\varphi_{k,1}=1$ in $B(y_k,R_k)$, and the second one $\varphi_{k,2}$ equal to 1 in $B(y_k,R_k'-\eta_n)^c$ and $0$ on $B(y_k,R'_k-2\eta_k)$ satisfying
 $$|\nabla \varphi_{k,1}|+|\nabla \varphi_{k,2}|\leq 1/\eta_k \to 0.$$

Next, we define
$$v_{k,1}=\varphi_{k,1} u_k \quad  \text{ and } \quad v_{k,2}=\varphi_{k,2} u_k.$$ 
We want  to prove that the sum  $v_{k,1}+v_{k,2}$ has almost the same $L^2$ norm as the original function $u_k$ because $w_k$ is in a dichotomy situation. Let us define the annulus $A_k:=B(y_k,R'_k)\setminus B(y_k,R_k)$. Because of \eqref{information1} and the fact that $\|w_k\|_2=1$ for all $k$, we directly get
$$\int_{A_k} |w_k|^2 \;dx \to 0,$$
and since $|u_k|=|w_k|$ we also have for $i=1,2$,
$$\int_{A_k} |u_k \varphi_{k,i}|^2 \;dx \leq \int_{A_k} |u_k |^2 \;dx= \int_{A_k} |w_k |^2 \;dx \to 0.$$
We deduce that 
\begin{eqnarray}
\int_{\mathbb{R}^N}|v_{k,1}|^2 \;dx  =\int_{B(y_k,R_k)}|w_{k}|^2 \;dx + \int_{A_k} |u_k  \varphi_{k,1}|^2 \to \alpha_1 \label{inin1}\\
\int_{\mathbb{R}^N}|v_{k,2}|^2 \;dx  =\int_{B(y_k,R_k')}|w_{k}|^2 \;dx + \int_{A_k} |u_k  \varphi_{k,2}|^2 \to 1-\alpha_1.\label{inin2}
\end{eqnarray}

Then we want to estimate the difference of the symmetrized gradients. For that purpose we compute $e(\varphi u)$ as follows. From the identity
$$\partial_j (\varphi u^i)=  u^i\partial_j\varphi +\varphi \partial_j u^i,$$ 
we get
\begin{eqnarray}
e(\varphi u)_{ij}&=&  \frac{ u^i\partial_j\varphi +\varphi \partial_j u^i}{2}+\frac{ u^j\partial_i\varphi +\varphi \partial_i u^j}{2} \notag \\
&=& \frac{u^i\partial_j\varphi +u^j\partial_i\varphi }{2}+\varphi e(u)_{i,j}.
\end{eqnarray}
Therefore,  pointwisely in $\mathbb{R}^N$ it holds the following estimate
$$|e(u_k \varphi_{k,1})| \leq  |u| |\nabla \varphi_{k,1}| +  |\varphi_{k,1}||e(u_k)|\leq \frac{1}{\eta_k} |u| + |e(u_{k})|$$
and the same holds true for $v_{k,2}$,
$$|e(u_k \varphi_{k,2})| \leq  \frac{1}{\eta_k} |u| + |e(u_{k})|.$$
Taking the square we get, for $i=1,2$,
\begin{eqnarray}
|e(v_{k,i})|^2 \leq \frac{1}{\eta_k^2} |u_k|^2 + 2\frac{1}{\eta_k}|u_k||e(u_k)|+  |e(u_{k})|^2. \label{corona}
\end{eqnarray}

Now remember that $v_{k,1}$ and $v_{k,2}$ have disjoint support, and that  their sum  coincide with $u_k$ outside $A_k$, in which we can use \eqref{corona} to estimate 
\begin{eqnarray}
\int_{\mathbb{R}^N} |e(u_k)|^2 \; dx - \int_{\mathbb{R}^N} |e(v_{k,1})|^2  \; dx &-& \int_{\mathbb{R}^N} |e(v_{k,2})|^2 \;dx \notag \\
 &\geq& -2\int_{A_k} \frac{1}{\eta_k^2} |u_k|^2 + 2\frac{1}{\eta_k}|u_k||e(u_k)| \;dx.
\end{eqnarray}
Since $\frac{1}{\eta_k}\to 0$ and both $u_k$ and $e(u_k)$ are uniformly bounded in $L^2$, we deduce that the term on the right-hand side converges to zero thus
\begin{eqnarray}
\liminf_{k \to +\infty}\int_{\mathbb{R}^N} |e(u_k)|^2 \; dx - \int_{\mathbb{R}^N} |e(v_{k,1})|^2  \; dx- \int_{\mathbb{R}^N} |e(v_{k,2})|^2 \;dx\geq 0. \label{liminf}
\end{eqnarray}
This allows to compare the Rayleigh quotient of $u_k$ with the one of $v_{k,1}+v_{k,2}$. More precisely,   using \eqref{liminf}, the standard inequality on real nonnegative numbers $a$, $b$ and positive numbers $c$, $d$,
$$\frac{a+b}{c+d}\geq \min\left\{\frac{a}{c}, \frac{b}{d}\right\},$$
and also \eqref{inin1} and \eqref{inin2}, we obtain  
\begin{eqnarray}
\lambda^*:=\inf_{|\Omega|\leq V}\Lambda(\Omega)&=&\lim_{k\to +\infty}\int_{\R^N} |e(u_k)|^2 \;dx \notag \\
&\geq& \liminf \int_{\mathbb{R}^N} |e(v_{k,1})|^2 \;dx+ \int_{\mathbb{R}^N} |e(v_{k,2})|^2 \;dx \notag \\
&=&\liminf \frac{\int_{\mathbb{R}^N} |e(v_{k,1})|^2+ \int_{\mathbb{R}^N} |e(v_{k,2})|^2}{\int_{\mathbb{R}^N} |v_{k,1}|^2 \; dx+ \int_{\mathbb{R}^N} |v_{k,2}|^2 \; dx} \label{return} \\
&\geq & \min \left\{ \liminf \frac{\int_{\mathbb{R}^N} |e(v_{k,1})|^2}{\int_{\mathbb{R}^N} |v_{k,1}|^2 \; dx}  \;, \;\liminf \frac{ \int_{\mathbb{R}^N} |e(v_{k,2})|^2}{ \int_{\mathbb{R}^N} |v_{k,2}|^2 \; dx}   \right\}. \label{yannickQ}
\end{eqnarray}
Notice that applying the concentration principle on the sequence $v_k^1$, we obtain that $v_k^1$ is in the compactness situation, with concentration value $\alpha$. In particular, arguing as in the compactness case, we can assume that  $v_k^1$ converges strongly in $L^2$ (and weakly in $H^1$) to a function $v\in H^1(\mathbb{R}^N)$. Then, if the minimum above is achieved for $v_{k,1}$, we deduce that the quasi-open set $\Omega^*= \{|v| > 0\}$ is an optimal domain, and the proof is concluded from the compactness situation. So we have to consider that it is not the case.

But then it means that $v$, being the $L^2$ limit of $v_k^1$, satisfies 
$$\frac{\int_{\R^N} |\nabla v|^2 \;dx}{\int_{\R^N} |v|^2 \;dx} > \lambda^*,$$
or differently,
 $$ \int_{\R^N} |\nabla v|^2 \;dx> \alpha \lambda^*.$$
 We also know that 
 $$\liminf \frac{ \int_{\mathbb{R}^N} |e(v_{k,2})|^2}{ \int_{\mathbb{R}^N} |v_{k,2}|^2 \; dx}=\lambda^*,$$
because  by assumption the minimum in \eqref{yannickQ} is achieved with the sequence $v_{k,2}$, and since  $\lim_{k\to +\infty}  \int_{\mathbb{R}^N} |v_{k,2}|^2 \; dx=1-\alpha$ we deduce that
 $$\liminf_{k\to +\infty} \int_{\mathbb{R}^N} |e(v_{k,2})|^2 = (1-\alpha)\lambda^*.$$

%

Now returning back to \eqref{return}, we have obtained
$$\lambda^*\geq \liminf_{k\to +\infty} \frac{\int_{\mathbb{R}^N} |e(v_{k,1})|^2+ \int_{\mathbb{R}^N} |e(v_{k,2})|^2}{\int_{\mathbb{R}^N} |v_{k,1}|^2 \; dx+ \int_{\mathbb{R}^N} |v_{k,2}|^2 \; dx} =\frac{ \int_{\R^N} |\nabla v|^2 \;dx + \liminf_{k\to +\infty}  \int_{\mathbb{R}^N} |e(v_{k,2})|^2}{\alpha +  (1-\alpha)}>\lambda^*,$$
a contradiction. This achieves the proof of the Theorem.
\end{proof}


\subsection{Regularity}

The purpose of this section is  to prove that any quasi-open solution of our shape optimisation problem, is actually an open set.  We will achieve this conclusion only for $N=2$ or $N=3$. The reason is that we need an apriori $L^p$ bound on an eigenfunction which, up to our knowledge, is not know in any dimension (see also Remark \ref{bound}) below.

Here is a general regularity result valid in any dimension.

\begin{theorem} \label{open} Let $\Omega^* \subset \R^N$ be a quasi-open  solution to the problem

$$\min_{\Omega \in \mathcal{O} \; \text{ such that }\; |\Omega| \leq V} \Lambda(\Omega).$$

Assume moreover that $u \in L^p(\R^N)$ with $p >N$. Then  $u\in \mathscr{C}^{0,\alpha}(\R^N)$,  for all  $\alpha<1-\frac{N}{p}$.  As a consequence, $\Omega^*$ is an open set. 
\end{theorem}

In particular in dimension $N=2$  and $N=3$ we obtain the following.

\begin{corollary} Assume that the dimension  $N=2$ or $N=3$. Then for all $V>0$ there exists an open solution $\Omega^*$ for the problem
$$\min \left\{ \Lambda(\Omega) \; , \; \Omega\subset \R^N \text{ open set such that }\; |\Omega| \leq V \right\}.$$
Moreover, any eigenfunction associated with $\Lambda(\Omega^*)$   belongs to $\mathscr{C}^{0,\alpha}(\R^N)$ for all $\alpha<1$ if $N=2$ and for all $\alpha<\frac{1}{2}$ if $N=3$.
 \end{corollary}
 
\begin{proof} Let $u$ be an eigenfunction associated with $\Lambda(\Omega^*)$. Since $u\in H^1(\R^N)$, by the Sobolev embedding we know that $u \in L^{p}(\R^N)$ with $p$ arbitrary large for $N=2$ or $p=2^*=\frac{2N}{N-2}$ for $N>2$.  Then by Proposition \ref{quasiminimal} below
we know that $u$ is a Lam\'e quasi-minimizer with exponent $\gamma=N-\frac{2N}{p}$. To conclude that $u\in \mathscr{C}^{0,\alpha}$ we would  need 
that $p>N$. This is true for $N=2$ or $N=3$. For $N=2$ we deduce from  Proposition \ref{aga}  that $u \in \mathscr{C}^{0,\alpha}$ for all $\alpha<1$. If $N=3$ we deduce, from  Proposition \ref{aga},  that $u \in \mathscr{C}^{0,\alpha}$ for all $\alpha <\frac{1}{2}$.
\end{proof}

 \begin{remark} \label{bound} Let us  stress that the conclusion of Theorem \ref{open} does not imply that $u\in L^\infty(\R^N)$. In other words by $u\in \mathscr{C}^{0,\alpha}(\R^N)$ we mean, that for a representative of $u$ it holds    $|u(x)-u(y)|\leq C|x-y|^\alpha$ which is enough to prove that $u$ is continuous.  Since $\Omega^*$ may not be a bounded set, we do not conclude that $u$ is bounded. Let us mention that in the scalar case it is well known that any eigenfunction associated to the first eigenvalue of the Dirichlet Laplacian belongs to $L^\infty(\R^N)$ together with the following nice bound, for which one usually refers to \cite[Example 2.1.8]{zbMATH00194234}: 
$$\|u\|_{L^\infty}\leq e^{\frac{1}{8\pi}}\lambda_1^D(\Omega)^{\frac{N}{4}} \|u\|_2.$$
It would be very interesting to know whether a similar bound is true for the Lam\'e eigenfunctions. This would directly imply the existence of an open solution in any dimension.
\end{remark}

The strategy of proof for Theorem \ref{open} is inspired by the seminal paper of Brian\c{c}on, Hayouni and Pierre \cite{briancon}, also declined later in different directions, see for instance \cite{lamboleybri,mazzoleni}. The general approach involves showing that a solution to the original problem is also a solution to a penalized version of the problem. We then exploit the regularity theory for free-boundary-type problems to conclude that the eigenfunction is globally Hölder continuous. In our case, however, this strategy requires a non-trivial adaptation due to the presence of the symmetrized gradient.
For instance in \cite{briancon}, the first step is to use a truncated test function and the co-area formula, which are not available for the symmetric gradient, thus in our context we have to argue differently from \cite{briancon}. In particular we are not able to arrive up to Lipschitz regularity but merely continuous regularity, which is enough to conclude that the optimal set is open.
\subsubsection{Equivalence with a penalized problem}


\begin{proposition}  \label{penalization}Let $V>0$ be given and $u$ be a solution for the problem 
\begin{eqnarray}
\lambda_V:=\min \left\{ \int_{\R^N}|e(u)|^2 dx   \quad  {\text s.t. } \quad  u \in H^1(\R^N), \; \int_{\R^N} |u|^2=1, \;\text{ and } \; |\{|u|>0\}|\leq V \right\}. \label{problemu}
\end{eqnarray}
Then for all $\lambda>\frac{\lambda_V}{V}$, and  all $v\in H^1(\R^N)$ we have
\begin{eqnarray}
\int_{\R^N}|e(u)|^2 \; dx \leq \int_{\R^N} |e(v)|^2 dx +\lambda_V\left(1-\int_{\R^N} |v|^2\right)^+ + \lambda \Big( |\{|v|>0\}|-V\Big)^+. \label{problemlam}
\end{eqnarray}
\end{proposition}
\begin{proof}For $v\in H^1(\R^N)$ and $\lambda>0$ we introduce
$$F_\lambda(v):=\int_{\R^N} |e(v)|^2 dx +\lambda_V\left(1-\int_{\R^N} |v|^2\right)^+ + \lambda\Big( |\{|v|>0\}|-V\Big)^+.$$
We first notice that, arguing as in Theorem \ref{existence}, $F_\lambda$ admits a minimizer $u_\lambda \in H^1(\R^N)$.   

Our next goal is to prove that for $\lambda$ large enough, then $|\{|u_\lambda|>0\}|\leq V$.  Assume for a contradiction that  $|\{|u_\lambda|>0\}|> V$. In the sequel we will write $\Omega:=\{|u_\lambda|>0\}$. Then 
we compare $u_\lambda$ with the competitor $v:= u_\lambda(t x)$ with the choice
$$t:=\left(\frac{|\Omega|}{V}\right)^{\frac{1}{N}}\geq 1.$$
Since $u_\lambda \in H^1_0(\Omega)$, it follows that $v\in H^1_0(\frac{1}{t}\Omega)$. Moreover,
$$|\{|v|>0\}|=|\frac{1}{t}\Omega|=\frac{1}{t^N}|\Omega|=V.$$
 Next, we use that $u_\lambda$ is a minimizer thus
$$F_{\lambda}(u_\lambda) \leq F_\lambda(v),$$
which yields in particular,
\begin{eqnarray}
\int_{\R^N}|e(u_\lambda)|^2 \;dx + \lambda\big( |\{|u_\lambda|>0\}|-V\big)^+\leq  F_\lambda(u_\lambda)\leq \int_{\Omega} |e(v)|^2 \; dx +\lambda_V\left(1-\int_{\R^N} |v|^2\right)^+  \label{test1}
 \end{eqnarray}
 because
$$\big( |\{|v|>0\}|-V\big)^+=0.$$
Now notice that 
$$\int_{\Omega} |e(v)|^2 \; dx =t^{2-N}  \int_{\R^N} |e(u_\lambda)|^2 \;dx= \left(\frac{|\{|u_\lambda|>0\}|}{V}\right)^{\frac{2-N}{N}} \int_{\R^N} |e(u_\lambda)|^2 \;dx \leq  \int_{\R^N} |e(u_\lambda)|^2 \;dx,$$ 
and 
$$\int_{\R^N} |v|^2=t^{-N} \int_{\R^N} |u_\lambda|^2=\frac{V}{|\{|u_\lambda|>0\}|}.$$
Therefore, we deduce from \eqref{test1}   that 
\begin{eqnarray}
 \lambda \Big( |\{|u_\lambda|>0\}|-V\Big)^+
 &\leq& \lambda_V\left(1-\frac{V}{|\{|u_\lambda|>0\}|}\right)^+. \label{test2}
 \end{eqnarray}
But then we obtain a bound $\lambda\leq M_0$, where 
 $$M_0=\max_{s \geq V}\left(\lambda_V \frac{1}{s-V}\left(1-\frac{V}{s}\right)\right) =\max_{s \geq V} \frac{\lambda_V}{s}= \frac{\lambda_V}{V}.$$
We arrive to the conclusion that  for $\lambda >\frac{\lambda_V}{V}$, then we necessarily have
$$|\{|u_\lambda|>0\}|\leq V.$$

Now let us fix $\lambda >\frac{\lambda_V}{V}$, and finish the proof of the Proposition. We denote by $u$ the minimizer for the problem \eqref{problemu}, and we pick any $v\in H^1(\R^N)$. From the inequality 
$$F_\lambda(u_\lambda)\leq F_\lambda(u),$$
and the fact that   $|\{|u>0\}|\leq V$ and  $\int_{\R^N}|u|^2=1$, we deduce that 

$$F_\lambda(u_\lambda)\leq \int_{\R^N} |e(u)|^2 dx.$$

On the other hand by the definition of $\lambda_V$, we know that 
$$\int_{\R^N}|e(u_\lambda)|^2 \;dx -\lambda_V \int_{\R^N} |u_\lambda|^2 \;dx \geq 0,$$
so that
$$F_\lambda(u)=\int_{\R^N}|e(u)|^2 \;dx =\lambda_V  \leq \int_{\R^N}|e(u_\lambda)|^2 \;dx + \lambda_V\left(1- \int_{\R^N} |u_\lambda|^2 \;dx\right) \leq F_\lambda(u_\lambda),$$
where for the last inequality we have used that $|\{|u_\lambda|>0\}|\leq V$.
All together we have proved that 
$$F_\lambda(u)=F_\lambda(u_\lambda),$$
and therefore $u$ is also a minimizer of $F_\lambda$. But then \eqref{problemlam} holds true for every $v\in H^1(\R^N)$ because $F_\lambda(u)=\int_{\R^N} |e(u)|^2\;dx$ and $F_\lambda(u)\leq F_\lambda(v)$. This achieves the proof of the proposition.
 \end{proof}

\subsubsection{Lam\'e quasi-minimizers}

In order to investigate the regularity properties of an optimal domain we introduce the following definition. 

\begin{definition}[Quasi-minimizer] \label{defQuasi}We say that $u \in H^1(\R^N)^N$ is a  quasi-minimizer  for the Lam\'e energy if and only if   $u$ satisfies the following minimality property: there exists $C>0$ and  $\gamma>0$ such that  for all ball $B_r\subset \R^N$ of radius $r \in (0,1)$ and for all $v\in H^1(\R^N)^N$ such that $u=v$ on $\R^N\setminus B_r$ we have
\begin{eqnarray}
\int_{B_r}|\nabla u|^2 +({\rm div}(u))^2  \; dx \leq \int_{B_r} |\nabla v|^2 +({\rm div}(v))^2dx +   Cr^\gamma. \label{problemlam2}
\end{eqnarray}
\end{definition}

The definition is motivated by the following observation.

\begin{proposition}\label{quasiminimal}Let  $u$ be a solution for the problem 
\begin{eqnarray}
\min \left\{ \int_{\R^N}|e(u)|^2 dx   \quad  {\text s.t. } \quad  u \in H^1(\R^N)^N, \; \int_{\R^N} |u|^2=1, \;\text{ and } \; |\{|u|>0\}|\leq V \right\}. \label{problemu2}
\end{eqnarray}
Assume moreover that $u\in L^p(\R^N)^N$ with $p\geq 2$. Then $u$ is a quasi-minimizer for the Lam\'e energy in the sense of Definition~\ref{defQuasi} with $\gamma=N-\frac{2N}{p}$ .
\end{proposition}  

\begin{proof}By  Proposition \ref{penalization} we already know that $u$ satisfies the following minimality property: for all $v\in H^1(\R^N)^N$ we have
\begin{eqnarray}
\int_{\R^N}|e(u)|^2 \; dx \leq \int_{\R^N} |e(v)|^2 dx +\lambda_V\left(1-\int_{\R^N} v^2\right)^+ + \lambda \Big( |\{|v|>0\}|-V\Big)^+,\label{problemlam2ter}
\end{eqnarray}
or equivalently, using \eqref{kornkornkorn}, 

\begin{eqnarray}
\frac{1}{2}\int_{\R^N}|\nabla u|^2 + ({\rm div}(u))^2 \; dx \notag & \leq& \frac{1}{2}\int_{\R^N} |\nabla v|^2 + ({\rm div}(v))^2\, dx  \notag \\
&&+\lambda_V\left(1-\int_{\R^N} v^2\right)^+ + \lambda \Big( |\{|v|>0\}|-V\Big)^+.\label{problemlam2bis}
\end{eqnarray}

Then let $v$ be equal to $u$ outside $B_r$, so that the volume $|\{|v|>0\}|$ has at most increased by $\omega_Nr^N$, and since $\int_{\R^N} |u|^2 \;dx=1$, we can compute
$$1-\int_{\R^N} |v|^2 =1- \int_{\R^N}|u|^2 + \int_{\R^N} |u|^2 - \int_{\R^N} |v|^2 =\int_{B_r} |u|^2 - \int_{B_r} |v|^2. $$

 In other words from the minimality of $u$ we obtain
\begin{eqnarray}
\int_{B_r}|\nabla u|^2 + ({\rm div}(u))^2  \; dx  \leq  \int_{B_r} |\nabla v|^2 + ({\rm div}(v))^2dx +C\left(  \int_{B_r} |u|^2 - \int_{B_r} |v|^2  \right)^+ + Cr^{N}.\notag
\end{eqnarray}
If moreover $u \in L^p(\R^N)^N$ with $p\geq 2$, then denoting by $q$ the exponent satisfying $2q=p$ we can estimate
$$\left(  \int_{B_r} |u|^2 - \int_{B_r} |v|^2  \right)^+\leq  \int_{B_r} |u|^2 \leq |B_r|^{\frac{1}{q'}}\left(\int_{B_r} |u|^{2q}\right)^{\frac{1}{q}} =C r^{\frac{N}{q'}}\|u\|_p^{\frac{p}{q}}=Cr^{\gamma}$$
  with $\gamma=\frac{N}{q'}=N-\frac{2N}{p}$. Since $\gamma <N$ and $r\leq 1$, it follows that  $Cr^N\leq Cr^{\gamma}$ and finally $u$ is a  quasi-minimizer for the Lam\'e energy in the sense of Definition~\ref{defQuasi}  with $\gamma=N-\frac{2N}{p}$.
\end{proof}

Now  Theorem \ref{open} follows from gathering together Proposition \ref{quasiminimal} with the following one after noticing that the condition $\gamma> N-2$ with $\gamma=N-\frac{2N}{p}$ is equivalent to  $p>N$. 

\begin{proposition} \label{aga} Let $u \in H^1(\R^N)^N$ be a quasi-minimizer for the Lam\'e energy with exponent $\gamma \in (N-2,N]$. Then  $u\in C^{0,\alpha}(\R^N)^N$,  for all $\alpha<\frac{\gamma-(N-2)}{2}$.
\end{proposition}  

\begin{proof} The proof is inspired by standard arguments in free boundary theory, such as for instance \cite[Theorem 2.1.]{david_toro}. The novelty here is that we have to take care of the Lam\'e energy instead of the standard Dirichlet energy. However, since the Lam\'e system is elliptic 
in the sense of systems, (which means that it satisfies the strong Legendre-Hadamard ellipticity condition),  we can conclude by use of similar techniques from the regularity theory for elliptic equations.

In this proof we will keep denoting by $C>0$ a universal constant that could change from line to line. Let $B_r\subset \R^N$ be a given ball of radius $r\in (0,1)$ and let $v$ be the solution for the problem
$$
\min_{v\in u+H_0^1(B_r)}\Big \{   \int_{B_r} |\nabla (v)|^2 +({\rm div}(v))^2 \;dx \Big\}.
$$
In other words $v$ is the replacement of $u$ in $B_r$, by a function satisfying $v=u$ on $\partial B_r$ and solution for the homogeneous Lam\'e system
\begin{eqnarray}
-\Delta v - \nabla ({\rm div} v)=0 \text{ in } B_r. \label{lameB}
\end{eqnarray}
Since the Lam\'e system satisfies the strong Legendre-Hadamard ellipticity condition, then $v$ enjoys some nice decaying properties. Indeed by applying standard regularity theory for elliptic systems (see for instance \cite[Theorem 4.11]{giaquinta}) we know that 
\begin{eqnarray}
 \sup_{B_{r/2}}|\nabla v|^2 \leq C \frac{1}{|B_r|}\int_{B_r}|\nabla v|^2 \;dx . \label{estimate}
 \end{eqnarray}
Let now $Q_s(v)$ be the quadratic form defined by
$$Q_s(v):=  \int_{B_s} |\nabla v|^2 +({\rm div}(v))^2 \;dx.$$
Then for $s\leq r/2$, using \eqref{estimate} we get 
\begin{eqnarray}
Q_s(v)=  \int_{B_s} |\nabla v|^2 +({\rm div}(v))^2 \;dx &\leq &(1+N^2)\sup_{B_{r/2}}| \nabla v|^2 |B_s| \notag \\
& \leq & C \left(\frac{s}{r}\right)^{N}\int_{B_r}|\nabla v|^2 \;dx, \notag \\
&\leq & C \left(\frac{s}{r}\right)^{N} Q_r(v), \label{decay}
\end{eqnarray}
where $C$ depends only on dimension $N$.

Moreover, the weak formulation of \eqref{lameB} says that for all $\varphi \in H^1_0(B_r)^N$,
$$ \int_{B_r}\nabla v : \nabla \varphi+ ({\rm div} v)({\rm div \varphi})\;dx=0.$$
In other words, if $A_r(u,v)$ denotes the   bilinear form associated with $Q_r$ and defined by
$$A_r(u,v):=  \int_{B_r} \nabla u:\nabla v +{\rm div}(u){\rm div}(v) \;dx,$$
we have $A_r(v,\varphi)=0$ for all $\varphi \in H^1_0(B_r)^N$. This applies in particular to $\varphi=u-v\in H^1_0(B_r)^N$ and we deduce from Pythagoras equality that 
$Q_r(u-v)+Q_r(v)=Q_r(u)$ or differently,
\begin{eqnarray}
Q_r(u-v)=Q_r(u)-Q_r(v). \label{pythagore}
\end{eqnarray}
 We will use this property later. 
 
 Notice also that $Q_s$ is a nonnegative quadratic form for any $s>0$ and using that $Q_s(b-a)\geq 0$ we obtain, for arbitrary $a,b$,
  $$2|A_s(a,b)|\leq Q_s(a)+Q_s(b),$$
  so that, for all $s<r/2$,
$$Q_{s}(u)=Q_{s}(u-v+v)\leq 2Q_{s}(u-v)+2Q_{s}(v).$$
Then using \eqref{decay} we arrive to 
\begin{eqnarray}
Q_{s}(u) &\leq& 2Q_{s}(u-v)+2Q_{s}(v)\leq C \left(\frac{s}{r}\right)^{N} Q_r(v) +2 Q_r(u-v) \notag\\
&\leq &   C \left(\frac{s}{r}\right)^{N} Q_r(u) +2 Q_r(u-v), \notag 
\end{eqnarray}
where for the last line we have used that $v$ is a minimizer of $Q_r$ and $u$ is a competitor.  Now we recall that $Q_r(u-v)=Q_r(u)-Q_r(v)$ (by \eqref{pythagore}) and we use that $u$ is a quasi-minimizer, so that 
$$Q_r(u-v)=Q_r(u)-Q_r(v)\leq  Cr^\gamma.$$
All in all, we have proved that for all $s\leq r/2$ we have 
\begin{eqnarray}
Q_{s}(u) \leq  C \left(\frac{s}{r}\right)^{N} Q_r(u)+Cr^\gamma. \label{decCay}
\end{eqnarray}
Of course we can assume $C\geq 2$. The decaying in \eqref{decCay} looks promising but we would prefer $s^\gamma$ on the last term instead of $r^\gamma$. We can obtain this up to decrease a bit the power $\gamma$ and use a technical dyadic argument.  This is standard (see for e.g. \cite[Lemma 5.6.]{filippo}) but let us write the full details for the reader's convenience.

Indeed, to lighten the notation we denote by $f(s)$ the non-increasing function $f:s\mapsto Q_s(u)$.
Let $a\in (0,1/2)$ be chosen later and let $r_k:=a^kr_0$. Let us prove by induction that for all $k \in \N$ it holds
\begin{eqnarray}
f(a^k r_0)\leq C^k a^{Nk}f(r_0)+ Ca^{(k-1)\gamma }r_0^\gamma \frac{C^k-1}{C-1}. \label{somme}
\end{eqnarray}
For  $k=0$ the inequality is obvious. Now let us assume that it holds true for some $k$. Then from the decaying property \eqref{decCay} we infer that (using in particular that $\gamma\leq N$ in \eqref{etaa}),
\begin{eqnarray}
f(a^{k+1} r_0)&\leq& Ca^N f(a^k r_0) + C (a^k r_0)^\gamma \notag \\
&\leq & Ca^N \Big( C^k a^{Nk}f(r_0)+  Ca^{(k-1)\gamma}r_0^\gamma \frac{C^k-1}{C-1}  \Big)+ C (a^k r_0)^\gamma \notag \\   
&\leq & C^{k+1} a^{N(k+1)} f(r_0) + Ca^{\gamma k}r_0^\gamma \frac{C^{k+1}-C}{C-1}  + C (a^k r_0)^\gamma  \label{etaa} \\
&= &  C^{k+1} a^{N(k+1)} f(r_0) +  Ca^{\gamma k}r_0^\gamma  \left(\frac{C^{k+1}-1}{C-1}\right),\notag
\end{eqnarray}
which proves \eqref{somme}. To simplify a bit we can write  it differently, taking into account that $C\geq 2$, 
\begin{eqnarray}
f(r_k)\leq C^k \left(\frac{r_k}{r_0}\right)^{N}f(r_0)+ C^{k+1}a^{-1}r_k^\gamma. \label{decayingbis}
\end{eqnarray}
The nice thing with \eqref{decayingbis} is that we have now $r_k^\gamma$ on the last term (compare with \eqref{decCay}), but the price to pay is the $C^k$ in factor. We will beat this factor by choosing well the constant $a$, and decreasing a bit the powers $N$ and $\gamma$.

Indeed, let $\alpha \in (0,1)$ be given and let us fix
$$a:=\frac{1}{C^{\frac{1}{\alpha}}},$$
so that 
$$r_k^\alpha C^k = a^{\alpha k}r_0^\alpha C^k=r_0^\alpha.$$
Then from \eqref{decayingbis} we deduce that 
\begin{eqnarray}
f(r_k)&\leq &C^k \left(\frac{r_k}{r_0}\right)^{N-\alpha}\left(\frac{r_k}{r_0}\right)^{\alpha} f(r_0)+ C^{k+1}a^{-1}r_k^{\gamma-\alpha} r_k^{\alpha} \notag \\
&\leq &  \left(\frac{r_k}{r_0}\right)^{N-\alpha} f(r_0) + C' r_k^{\gamma-\alpha},
\end{eqnarray}
where 
$$C'=Ca^{-1}r_0^\alpha.$$
Now let $s\in (0,1/2)$ be given. There exists $k$ such that $r_{k+1} \leq s \leq r_k$. In particular, $r_k\leq \frac{1}{a}s$.

Moreover, $s\mapsto f(s)$ is non decreasing so 
\begin{eqnarray}
f(s)\leq f(r_k)&\leq&  \left(\frac{r_k}{r_0}\right)^{N-\alpha} f(r_0) + C' r_k^{\gamma-\alpha} \notag \\
&\leq& a^{-(N-\alpha)} \left(\frac{s}{r_0}\right)^{N-\alpha} f(r_0) + a^{-(\gamma-\alpha)}C'  s^{\gamma-\alpha}. \notag
\end{eqnarray}
In conclusion we have proved that there exists a constant $C>0$ (depending on $N$, $\alpha$, $\gamma$, $r_0$) such that  for all $s\leq 1/2$ we have 
$$f(s)\leq C s^{N-\alpha} f(r_0)+ Cs^{\gamma-\alpha}\leq C s^{\gamma-\alpha} f(r_0)+ Cs^{\gamma-\alpha},$$
where for the last inequality we have used $N\geq \gamma$. Returning back to $u$, and estimating $f(r_0)$ by $\int_{\R^N}|e(u)|^2 \;dx$, we conclude, using also the Poincar\'e inequality, that for all $r\leq 1/2$,
$$\int_{B_r}|u-m_u|^2 \; dx\leq Cr^2\int_{B_r} |\nabla u|^2 \;dx \leq C r^{2+\gamma -\alpha}=C r^{N+(2+\gamma-N -\alpha)},$$
where $m_u$ denotes the (vectorial) average of $u$. Remember that here $\alpha$ is arbitrary close to $0$. Then by standard results about Campanato spaces (see for e.g. \cite[Theorem 5.4]{filippo}), provided that 
$$ 2+\gamma-N -\alpha>0,$$
then
$u \in \mathscr{C}^{0,\beta}(\R^N)$ with $\beta=  \frac{2+\gamma-N -\alpha}{2}$. Since $\alpha$ is arbitrary, this means that $u$ belongs to    $\mathscr{C}^{0,\beta}(\R^N)$ for all $\beta<\frac{\gamma-(N-2)}{2}$, as soon as $\gamma>N-2$, and the Proposition follows. 
\end{proof}


\section{Characterization of minimizers}\label{secoptim}

 
\subsection{Optimality conditions: first and second orders}
We give first general formulae for the first and second order shape derivative of the Lam\'e eigenvalue. We will then apply it to
get optimality conditions (assuming that the minimizer is smooth enough to justify our computations).
As kindly mentioned by D. Buoso, these computations (and the criticality of the ball) already appeared in the review paper
\cite{Buoso-Lamberti15}. Moreover, they use weaker regularity assumptions on the domain $\Omega$ in this paper.
For sake of completeness, we give the main results here that will be useful in Section \ref{secdisk}.
 Let $\Omega$ be a bounded domain with $\mathscr C^3$ boundary. This regularity assumption allows us to ensure that all quantities we will handle belong to $L^2(\partial\Omega)$.
Let ${V}\in W^{4,\infty}(\R^N,\R^N)$ and introduce, for any $t\in (-1,1)$ small enough, 
\[ 
\Omega_{t{V}}:=(\mathrm{Id}+t{V})\Omega.
\] 
Recall that, for any $t$ small enough, $(\mathrm{Id}+t\Phi)$ is a smooth diffeomorphism. 

 \paragraph{First order optimality conditions.}
 
For a given shape functional $F$ and a given shape $\Omega$, we say that $F$ is differentiable at $\Omega$ if, for any ${V}\in  W^{4,\infty}(\R^N,\R^N)$ compactly supported, the limit
\[ \langle dF(\Omega),{V}\rangle:=\lim_{t\to 0}\frac{F(\Omega_{t{V}})-F(\Omega)}t\] exists and if it is a linear form in ${V}$. In this case, this limit is called the first-order shape derivative of $F$ at $\Omega$ in the direction ${V}$.
 
 We consider the case of the general eigenvalue
 $$
 \Lambda(\Omega)=\inf_{u\in (H^1_0(\Omega))^N}\frac{\mu \int_\Omega |\nabla u|^2\, dx+(\lambda+\mu)\int_\Omega (\operatorname{div}(u))^2\, dx}{\int_\Omega |u|^2\, dx}.
 $$
 Let us assume moreover that $\Omega$ is such that $\Lambda(\Omega)$ is simple.
The associated PDE solved by the minimizer $u$ is
 \begin{equation}\label{pdeLame1}
 \left\lbrace
 \begin{array}{cc}
 -\mu \Delta u -(\lambda+\mu) \nabla ({\rm div}(u)) = \Lambda u & \mbox{ in } \Omega\\
 u=0  & \mbox{ on } \partial\Omega .\\
 \end{array}
 \right.
 \end{equation}
It is standard that $W^{4,\infty}(\R^N,\R^N)\ni {V}\mapsto {u}_{(\operatorname{Id}+{V})}\in [H^1(\Omega)]^N$ is differentiable, 
see for example  \cite{Buoso-Lamberti15}, \cite{zbMATH06838450}. Its Eulerian derivative $\dot{{u}}_{{V}}$, solves 
\begin{equation} \label{Pb:derLame}
\begin{cases}
-\mu \Delta \dot{{u}}_{{V}}-(\lambda+\mu)\nabla \operatorname{div} \dot{{u}}_{{V}} 	=\dot \Lambda {u}_{{V}}+\Lambda (\Omega)\dot{{u}}_{{V}} & \textrm{in }\Omega \\
\dot{{u}}_{{V}}=-\nabla {{u}_\Omega} n  ({V}\cdot n) & \textrm{on }\partial \Omega\,, 
\\ 
\int_{\Omega}  {u}_\Omega\cdot \dot{{u}}_{{V}}=0.
\end{cases}
\end{equation}
Let us multiply the main equation by ${u}_\Omega$ and then integrate by parts. We get
$$
\mu \int_\Omega \nabla u_\Omega :\nabla \dot{u}_V+(\lambda+\mu)\int_\Omega \operatorname{div}\dot{u}_V \operatorname{div}{u}_\Omega=\Lambda(\Omega)\int_\Omega u_\Omega\cdot \dot{u}_V+\dot{\Lambda}\int_\Omega |{u}_\Omega|^2.
$$
Similarly, let us multiply the main equation \eqref{pdeLame1} solved by $u_\Omega$ by $\dot{u}_V$ and then integrate by parts. Using the boundary conditions, we get
\begin{eqnarray*}
\mu \int_\Omega \nabla u_\Omega :\nabla \dot{u}_V+(\lambda+\mu)\int_\Omega \operatorname{div}\dot{u}_V \operatorname{div}{u}_V&=& \Lambda(\Omega)\int_\Omega u_\Omega\cdot \dot{u}_V-\mu \int_{\partial \Omega}|(\nabla u_\Omega )n|^2(V\cdot n)\\
&& -(\lambda+\mu)\int_\Omega (\operatorname{div}u_\Omega)^2 V\cdot n,
\end{eqnarray*}
by using that $\nabla u_\Omega n \cdot n = \operatorname{div}u_\Omega$ on $\partial \Omega$.
Finally, combining the two identities above yields
 $$
 \dot{\Lambda}=-\mu \int_{\partial \Omega}|(\nabla u_\Omega) n|^2(V\cdot n) -(\lambda+\mu)\int_\Omega (\operatorname{div}u_\Omega)^2 V\cdot n.
 $$
 We have then obtained the following result.
   \begin{proposition}\label{Pr:DifferentiabilityFormulaeBis}
Let $\Omega$ denote a $\mathscr C^3$ domain such that $\Lambda(\Omega)$ is simple. Let ${u}_\Omega$ be its associated 
(normalized) first eigenfunction. For any ${V}\in W^{4,\infty}(\R^N,\R^N)$, the mapping $W^{4,\infty}(\R^N,\R^N)\ni {V}\mapsto \Lambda(\Omega_{{V}})$ is differentiable. Denoting by $\dot \Lambda$ its differential, the first order derivative of $\Lambda$ is 
\begin{equation}\label{Eq:D1Lambda}
 \dot{\Lambda}=\langle d\Lambda(\Omega),{V}\rangle=-\mu \int_{\partial \Omega}|(\nabla u_\Omega) n|^2(V\cdot n) -(\lambda+\mu)\int_{\partial \Omega} (\operatorname{div}u_\Omega)^2 V\cdot n.
 \end{equation} 
\end{proposition}
\begin{remark}[Shape gradient]
Observe that $\nabla u_\Omega \nu\cdot n = \operatorname{div}u_\Omega$ on $\partial \Omega$ and denote by $[\nabla u_{i,\Omega}]_\tau:= \nabla u_{i,\Omega}-\frac{\partial u_{i,\Omega}}{\partial nu} n$ the tangential part of the gradient $\nabla u_{i,\Omega}$. According to the result above, the shape gradient $\nabla \Lambda(\Omega)$ reads
\begin{equation}\label{eq1106}
\nabla \Lambda(\Omega)=-\mu |(\nabla u_\Omega) n|^2 -(\lambda+\mu) (\operatorname{div}u_\Omega)^2 
\end{equation}
and can be decomposed as:
$$
\nabla \Lambda(\Omega)=-\mu \sum_i |[\nabla u_{i,\Omega}]_\tau |^2 -(\lambda+2\mu) (\operatorname{div}u_\Omega)^2 .
$$
\end{remark}
\begin{corollary}
 Let $\Omega^*$ by a solution with $\mathscr{C}^3$ boundary of the extremal eigenvalue problem
 $$
 \min_{|\Omega|=V_0}\Lambda(\Omega)
 $$
such that $\Lambda(\Omega^*)$ is simple.
 Then, denoting by $u_{\Omega^*}$ any associated eigenfunction, 
 $$\mu |(\nabla u_\Omega^*) n|^2 -(\lambda+\mu) (\operatorname{div}u_\Omega^*)^2 $$
 is constant on $\partial\Omega^*$.
 \end{corollary}
\begin{proof}
This is a consequence of Proposition~\ref{Pr:DifferentiabilityFormulaeBis}. Indeed, since we work with a volume constraint, there
exists a Lagrange multiplier such that the shape gradient of the eigenvalue is proportional to the derivative of the volume, namely
$\int_{\partial\Omega^*} V\cdot n$ whence the result.
As a particular case, if we take $\mu>0$ and $\lambda=0$ we obtain that $|e(u_{\Omega^*}|$ in that case  is constant on the boundary.
\end{proof}

\paragraph{Second order optimality conditions.}
 According for instance to  \cite{zbMATH03656460}, for any $\mathscr C^3$ domain $\Omega$ such that $\Lambda(\Omega)$ is simple, the mapping $\Omega\mapsto \Lambda(\Omega)$ is twice shape differentiable at $\Omega$ in the following sense: for any compactly supported vector field $V\in W^{4,\infty}(\R^N,\R^N)$, the map $t\mapsto \Lambda\left((\mathrm{Id}+t\Phi)\Omega\right)$ is twice differentiable at $t=0$. We will use the notations 
\[   
\langle d\Lambda(\Omega),V\rangle :=f'(0),\qquad \langle d^2\Lambda(\Omega)V,V\rangle :=f''(0).
\]
Similarly, the mapping $\Omega\mapsto u_\Omega$ is twice shape-differentiable  at $\Omega$, where $u_\Omega$ is the first  normalized eigenfunction of System \eqref{pdeLame1} on $\Omega$, in the sense that the mapping $g:t\mapsto u_{(\mathrm{Id}+tV)\Omega}$ is twice differentiable at $t=0$. We let $\dot{u}_V$ be its first order derivative at $t=0$ (often called Eulerian derivative in the standard shape optimization literature).
%
%

\begin{proposition}\label{Pr:DifferentiabilityFormulae}
For any $\mathscr C^3$ domain $\Omega$ such that $\Lambda(\Omega)$ is simple, let $u_\Omega$ be its associated first eigenfunction. For any $V\in W^{4,\infty}(\R^N,\R^N)$ compactly supported, the shape derivative $\dot{u}_V$ solves the PDE 
\begin{equation} \label{Pb:der}
\begin{cases}
-\mu \Delta \dot{u}_V-(\lambda+\mu)\nabla \operatorname{div}\dot{u}_V=\Lambda(\Omega)\dot{u}_V+ \langle d\Lambda(\Omega),V\rangle  {u_\Omega} & \textrm{in }\Omega\\
\dot{u}_V=-\nabla {u_\Omega} n (V\cdot n) & \textrm{on }\partial \Omega\,, 
\\ \int_{\Omega} u_\Omega\cdot \dot{u}_V=0.
\end{cases}
\end{equation}
 If, in addition, the vector field $V$ is normal to $\partial \Omega$, meaning that $V=(V\cdot n) n$ on $\partial\Omega$, the second-order shape derivative of $\Lambda$ at $\Omega$ is given by 
 \begin{eqnarray}\label{Eq:D2Lambdabis}
   \langle d^2\Lambda(\Omega)V,V\rangle &=& -\mu\int_{\partial \Omega}\left(H \left| \frac{\partial  u_\Omega}{\partial n}\right|^2+2\frac{\partial^2  u_\Omega }{\partial n^2}\cdot \frac{\partial  u_\Omega }{\partial n}\right)(V\cdot n)^2 \nonumber  \\ 
&& -(\lambda+\mu) \int_{\partial \Omega}\left(H (\operatorname{div} u_\Omega)^2+\frac{\partial (\operatorname{div} u_\Omega)^2}{\partial n}\right)(V\cdot n)^2\nonumber \\
&& -2\Lambda \int_\Omega |\dot{u}_V|^2+2\mu \int_\Omega |\nabla \dot{u}_V|^2+2(\lambda+\mu)\int_\Omega (\operatorname{div}\dot{u}_V)^2
\end{eqnarray} 
where $H$ is the mean curvature of $\partial \Omega$.

Furthermore, when $N=2$, one has
$$
\frac{\partial (\operatorname{div} u_\Omega)^2}{\partial n}=-\frac{2\mu}{\lambda+\mu}\operatorname{div}u_\Omega \left(H\frac{\partial u_\Omega}{\partial n}\cdot n+\frac{\partial^2u_\Omega}{\partial n^2}\cdot n \right)\qquad \text{on }\partial\Omega .
$$ 
\end{proposition}
\begin{proof}
Let us denote by $u_\Omega=(u_1,\dots, u_N)^\top$ the solution of \eqref{pdeLame1}.
General formulae for the shape differentiation of Dirichlet boundary value problem yield that $\dot{u}_V$ solves \eqref{Pb:der}, we refer to \cite[Chapter 5]{zbMATH06838450} for the detailed computations. 
%

Let us apply the Hadamard formula for integrals on variable boundaries \cite[Proposition 5.4.18]{zbMATH06838450} to \eqref{Eq:D1Lambda}. This yields
\begin{eqnarray*}
   \langle d^2\Lambda(\Omega)V,V\rangle &=&-2\mu \int_{\partial \Omega}  (\nabla \dot{u}_V)n\cdot (\nabla u_\Omega)n  (V\cdot n) -\mu\int_{\partial \Omega}\left(H | (\nabla u_\Omega)n|^2+\frac{\partial |(\nabla u_\Omega)n|^2}{\partial n}\right)(V\cdot n)^2
\\&& -2(\lambda+\mu)\int_{\Omega}\operatorname{div} u_\Omega \operatorname{div} \dot{u}_V (V\cdot n)\\
&& -(\lambda+\mu) \int_{\partial \Omega}\left(H (\operatorname{div} u_\Omega)^2+\frac{\partial (\operatorname{div} u_\Omega)^2}{\partial n}\right)(V\cdot n)^2+R,
\end{eqnarray*}
where $H$ denotes the mean curvature of $\partial\Omega$ and
\begin{eqnarray*}
R&=&-2\mu\int_{\partial \Omega}\langle (\partial_n u_\Omega)n',(\partial_nu_\Omega)n\rangle ( V\cdot n)-\mu\int_{\partial \Omega} \Vert (\partial_n u_\Omega )n\Vert^2( V\cdot n)'\\
&&-(\lambda+\mu)\int_{\partial\Omega}(\operatorname{div}u_\Omega)^2(V\cdot n)'
\end{eqnarray*}
where $n'$ is the Eulerian derivative of $n$.
The expression of \( d^2 \Lambda(\Omega) \) is independent of the specific extension chosen for \( n \), which allows us to consider a regular extension of \( n \), that  is unitary in a neighborhood of \( \partial\Omega \), without loss of generality. As a consequence, $n'=-\nabla_\Gamma(V\cdot n)$, the notation $\nabla_\Gamma$ standing for the tangential gradient.
First, note that \( ( V\cdot n)' = 0 \) since we are dealing with vector fields $V$ that are normal to \( \partial\Omega \). Furthermore, because \( u_\Omega = 0 \) on \( \partial\Omega \) and \( n' \) is orthogonal to \( n \), it follows that \( \nabla u_i \cdot n' = 0 \) for 
\( i \in \{1, 2, \ldots N\} \), which implies \( (\nabla u_\Omega)n' = 0 \). From this, we deduce that \( R = 0 \). 

Let us multiply the main equation of \eqref{Eq:D2Lambdabis} by $\dot{u}_V$ and then integrate by parts. We obtain
\begin{eqnarray*}
\lefteqn{\mu \int_\Omega |\nabla \dot{u}_V|^2+\mu \int_{\partial\Omega}(\nabla \dot{u}_V)n \cdot (\nabla u_\Omega)n (V\cdot n)+(\lambda+\mu)\int_\Omega (\operatorname{div}\dot{u}_V)^2}\\
&&+(\lambda+\mu)\int_{\partial\Omega}\operatorname{div}\dot{u}_V (\nabla u_\Omega) n\cdot n (V\cdot n)=\dot{\Lambda}\int_{\Omega}\dot{u}_V\cdot u_\Omega+\Lambda \int_\Omega |\dot{u}_V|^2.
\end{eqnarray*}
Using that $\int_{\Omega}\dot{u}_V\cdot u_\Omega=0$ and that
$$
(\nabla u_\Omega) n\cdot n=\sum_{i,j}\frac{\partial u_i}{\partial x_j}n_in_j=\sum_{i,j}\frac{\partial u_i}{\partial n}n_in_j^2=\sum_{i}\frac{\partial u_i}{\partial n}n_i=\sum_{i}\frac{\partial u_i}{\partial x_i}=\operatorname{div}u_\Omega
$$
on $\partial\Omega$, the equality above simplifies into
\begin{eqnarray*}
\lefteqn{\mu \int_{\partial\Omega}(\nabla \dot{u}_V)n \cdot (\nabla u_\Omega)n (V\cdot n)+(\lambda+\mu)\int_{\partial\Omega}\operatorname{div}\dot{u}_V \operatorname{div}u_\Omega (V\cdot n)} \\
&&=\Lambda \int_\Omega |\dot{u}_V|^2-\mu \int_\Omega |\nabla \dot{u}_V|^2-(\lambda+\mu)\int_\Omega (\operatorname{div}\dot{u}_V)^2.
\end{eqnarray*}
As a result, the second order derivative of $\Lambda$ rewrites
\begin{eqnarray*}
   \langle d^2\Lambda(\Omega)V,V\rangle &=& -\mu\int_{\partial \Omega}\left(H | (\nabla u_\Omega)n|^2+\frac{\partial |(\nabla u_\Omega)n|^2}{\partial n}\right)(V\cdot n)^2
\\&& -2\Lambda \int_\Omega |\dot{u}_V|^2+2\mu \int_\Omega |\nabla \dot{u}_V|^2+2(\lambda+\mu)\int_\Omega (\operatorname{div}\dot{u}_V)^2\\
&& -(\lambda+\mu) \int_{\partial \Omega}\left(H (\operatorname{div} u_\Omega)^2+\frac{\partial (\operatorname{div} u_\Omega)^2}{\partial n}\right)(V\cdot n)^2.
\end{eqnarray*}
Let us simplify the term
$$
A:=\frac{\partial |(\nabla u_\Omega)n|^2}{\partial n} .
$$
By expanding $|(\nabla u_\Omega)n|^2$, we get
$$
|(\nabla u_\Omega)n|^2=\sum_i \left(\frac{\partial u_i}{\partial n}\right)^2.
$$
Now since $u_i=0$ on $\partial \Omega$, one has 
$$
\frac{\partial}{\partial n}\left(\frac{\partial  u_i}{\partial n}\right)^2=2 \left(\frac{\partial  u_i}{\partial n}\right)
\left(\frac{\partial^2  u_i}{\partial n^2}\right)
$$
%
%
and therefore $A=2\frac{\partial u_\Omega}{\partial n}\cdot \frac{\partial^2 u_\Omega}{\partial n^2}$. 

Now we look at the term 
$$B:=\frac{\partial (\operatorname{div} u_\Omega)^2}{\partial n}.$$
To simplify this term, we will use the main equation in \eqref{pdeLame1}. 
We have
$$
\frac{\partial \operatorname{div}u_\Omega}{\partial n}= \sum_i \frac{\partial \operatorname{div}u_\Omega}{\partial x_i} \,n_i.
$$
According to  \eqref{pdeLame1} and the decomposition of the Laplacian 
$$
\Delta u_i=\Delta_\tau u_i+H\frac{\partial u_i}{\partial n}+\frac{\partial^2 u_i}{\partial n^2}\qquad \text{on }\partial \Omega,
$$
we get
$$
(\lambda+\mu)\frac{\partial \operatorname{div}u_\Omega}{\partial n}=-\mu \left(H\frac{\partial u_\Omega}{\partial n}\cdot n+\frac{\partial^2u_\Omega}{\partial n^2}\cdot n\right)\qquad \text{on }\partial\Omega ,
$$
and thus
$$
\frac{\partial (\operatorname{div} u_\Omega)^2}{\partial n}=-\frac{2\mu}{\lambda+\mu}\operatorname{div}u_\Omega \left(H\frac{\partial u_\Omega}{\partial n}\cdot n+\frac{\partial^2u_\Omega}{\partial n^2}\cdot n\right)\qquad \text{on }\partial\Omega ,
$$
whence the last claim of the theorem.
\end{proof}

 \subsection{Multiplicity of minimal eigenvalues}
 \begin{lemma}
 Assume that $N=2$ and let $\Omega^*$ be a minimizing domain with $\mathscr{C}^3$ boundary for the problem
 $$
 \min_{|\Omega|=V_0}\Lambda(\Omega).
 $$
 Then, $\Lambda(\Omega^*)$ is at most of multiplicity 2.
 \end{lemma}
 \begin{proof}
 In what follows, let us denote by $[y(x)]_\tau$ the tangential part of a vector field $y\in L^2(\partial\Omega,\R^2)$ at $x\in \partial\Omega$, in other words
 $$
 [y(x)]_\tau=y(x)-(y(x)\cdot n(x))n(x).
 $$
 Let us assume that $\Lambda(\Omega)$ has multiplicity $m\geq 3$. According to classical results for the derivative of multiple
 eigenvalues (see e.g. \cite{henrot06}, \cite{Buoso-Lamberti15},  \cite{CDM21}), the first order optimality conditions read:
let $V$ denote a smooth vector field, then the directional derivative of $|\Omega|\Lambda(\Omega)$ exists and it is the smallest eigenvalue of the $m\times m$ matrix $\mathcal{M}$ whose entries are
 $$V_0\mathcal{M}-\Lambda \int_{\partial\Omega} (V\cdot n)I_2, \quad \text{where }\mathcal{M}_{i,j}=-\int_{\partial\Omega} \left\lbrack \mu [\nabla u^i:\nabla u^j] + (\lambda+\mu) \operatorname{div} u^i \operatorname{div} u^j\right\rbrack V\cdot n,$$
where $(u^1,\dots,u^m)$ is an orthonormal basis of associated eigenfunctions. By minimality, this directional derivative has to be nonnegative.
Since we can take both $ V$ and $-V$, this shows  that $ \mathcal{M}=\frac{\Lambda}{V_0}  \int_{\partial\Omega} (V\cdot n) I_2$. In particular, 
$$
\mu [\nabla u^i:\nabla u^j] + (\lambda+\mu) \operatorname{div} u^i \operatorname{div} u^j=0\quad \text{on $\partial\Omega$, for }i\neq j,
$$
which rewrites
$$
\mu \frac{\partial u^i}{\partial n}\cdot \frac{\partial u^j}{\partial n} + (\lambda+\mu) \operatorname{div} u^i \operatorname{div} u^j=0\quad \text{on $\partial \Omega$ for }i\neq j,
$$
by using the Dirichlet boundary condition. Observe moreover that $\partial u^i/\partial n \cdot n=\sum_k \partial u^i_k/\partial n n_k=\operatorname{div}u^i$ and therefore, the condition above rewrites
$$
\mu \left[\frac{\partial u^i}{\partial n}\right]_\tau\cdot \left[\frac{\partial u^j}{\partial n}\right]_\tau + (\lambda+2\mu) \operatorname{div} u^i \operatorname{div} u^j=0\quad \text{on $\partial\Omega$ for }i\neq j,
$$
or equivalently
$$
\left(\sqrt{\mu} \left[\frac{\partial u^i}{\partial n}\right]_\tau+\sqrt{\lambda+2\mu} \operatorname{div}(u^i)n)\right)\cdot\left(\sqrt{\mu} \left[\frac{\partial u^j}{\partial n}\right]_\tau+\sqrt{\lambda+2\mu} \operatorname{div}(u^j)n)\right)=0\quad \text{on $\partial\Omega$ for }i\neq j.
$$
We have obtained a family of (at least) three orthogonal vectors in $\R^2$, and thus a contradiction. 
 \end{proof}

 \section{The case of the disk}\label{secdisk}
 Our first aim is to compute the first eigenvalue of the unit disk in $\R^2$. We recall that
  the Lam\'e coefficients $\lambda,\mu$ are such that $\mu>0, \lambda+\mu >0$. The first eigenvalue is then defined by
 \begin{equation}\label{LambdaLame}
\Lambda:= \min_{u=(u_1,u_2) \in H^1_0(\Omega)^2} \frac{ \mu\left(\int_{\Omega}|\nabla u_1|^2 \; dx+\int_{\Omega}|\nabla u_2|^2 \; dx\right) +(\lambda+\mu) \int_{\Omega} ({\rm div}(u))^2 \;dx}{\int_{\Omega}(u_1)^2 \; dx+\int_{\Omega}( u_2)^2 \; dx} 
 \end{equation}
 and the PDE solved by the minimizer $u=(u_1,u_2)$ is
 \begin{equation}\label{pdeLame}
 \left\lbrace
 \begin{array}{cc}
 -\mu \Delta u -(\lambda+\mu) \nabla ({\rm div}(u)) = \Lambda u & \mbox{ in } \Omega\\
 u=0  & \mbox{ on } \partial\Omega .\\
 \end{array}
 \right.
 \end{equation}
 
 \subsection{Eigenvalues and eigenvectors of the unit disk}\label{sec;eig_disk}
 We follow the strategy proposed in Capoferri et al, \cite{CFLV23}.  
 We will need some Helmholtz decomposition of the vector $u$. Let us state a more general Lemma that will be also useful for the derivative later.
 \begin{lemma}\label{lem:helm}
 Let $v=(v_1,v_2)$ be a smooth function satisfying the equation
 \begin{equation}\label{eqsatv}
 -\mu \Delta v -(\lambda + \mu) \nabla({\rm div}(v))=\Lambda v +f
 \end{equation}
  in a smooth domain $\Omega$ with a given function $f$.
There exist two function $\psi_1$ and $\psi_2$ in $C^\infty(\Omega)$ such that 
\begin{equation}\label{helmoltz}
v+\frac{1}{\Lambda}\,f=\nabla \psi_1+\operatorname{curl}\psi_2\quad \text{in }\Omega.
\end{equation}
Furthermore $\psi_1$ and $\psi_2$ respectively satisfy the PDE
$$
-(\lambda+2\mu)\Delta \psi_1=\Lambda \psi_1\quad \text{in }\Omega,
$$
and
$$
-\mu\Delta \psi_2=\Lambda \psi_2-\frac{\mu}{\Lambda}\operatorname{curl}f  \quad \text{in }\Omega .
$$
\end{lemma}

\begin{proof}
This proof has been suggested by M. Levitin.
Let us first set
$$
\psi_1=-\frac{\lambda+2\mu}{\Lambda}\operatorname{div}\ v
\quad\text{and}\quad
\psi_2=-\frac{\mu}{\Lambda}\operatorname{curl}v=-\frac{\mu}{\Lambda}(\partial_x v_2-\partial_y v_1).
$$
According to \eqref{eqsatv}, one has
\begin{eqnarray*}
v+\frac{f}{\Lambda} &=& \frac{1}{\Lambda}\left(-\mu \Delta v-(\lambda+\mu)\nabla \operatorname{div}(v)\right)\\
&=& \frac{1}{\Lambda}\left(-\mu \left(\Delta v-\nabla \operatorname{div}(v)\right)-(\lambda+2\mu)\nabla \operatorname{div}(v)\right).
\end{eqnarray*}
Note that
$$
\Delta  v-\nabla \operatorname{div}(v)=\begin{pmatrix}
\partial_{yy}\ v_1 - \partial_{xy} v_2\\
\partial_{xx} v_2 - \partial_{xy} v_1
\end{pmatrix}=\begin{pmatrix}
-\partial_{y}\\
\partial_{x}
\end{pmatrix}(\partial_x v_2-\partial_y v_1).
$$
Combining the identities above, we thus infer that $v$ satisfies \eqref{helmoltz}.

Now, observe that we can write the equation \eqref{eqsatv} as
$$\mu \operatorname{curl}\operatorname{curl}(v)-(\lambda+2\mu) \operatorname{grad}\operatorname{div}(v)= \lambda v +f$$
Now, passing to the divergence in this equation and using $\operatorname{div}\operatorname{curl}=0$
and $\operatorname{div}\operatorname{grad}\operatorname{div}=\Delta \operatorname{div}$ yields
$$
-(\lambda+2\mu)\Delta \psi_1=\Lambda \psi_1\quad \text{in }\Omega .
$$ 
In the same way, taking the curl in this equation and using $\operatorname{curl}\operatorname{grad}=0$
and $\operatorname{curl}\operatorname{curl}= - \Delta $
yields 
$$
-\mu\Delta \psi_2=\Lambda \psi_2-\frac{\mu}{\Lambda}\operatorname{curl} f , \quad \text{in }\Omega.
$$
The conclusion follows.
\end{proof}
 
Now, to compute the eigenvalues and eigenvectors of the unit disk, we use the decomposition provided by Lemma \ref{lem:helm}
(with $v=u$ and $f=0$),
 \begin{equation}\label{helm0}
 u=\nabla \psi_1 + \operatorname{curl}\psi_2=\left(\begin{array}{c}
 \frac{\partial \psi_1}{\partial x} \,+ \frac{\partial \psi_2}{\partial y} \\
  \frac{\partial \psi_1}{\partial y} \, -  \frac{\partial \psi_2}{\partial x} 
 \end{array}\right)
 \end{equation}
 and we use that the scalar potentials $\psi_i,i=1,2$ satisfy an Helmholtz equation 
 \begin{equation}\label{helm1}
 -\Delta \psi_i=\omega_i^2  \psi_i \quad \mbox{in } \Omega
 \end{equation}
 where
 $$\omega_1^2=\frac{\Lambda}{\lambda+2\mu},\quad \omega_2^2=\frac{\Lambda}{\mu}\,.$$
 We introduce
 $$\omega=\sqrt{\Lambda},\ a_1=\frac{1}{\sqrt{\lambda+2\mu}},\ a_2=\frac{1}{\sqrt{\mu}}$$
 therefore, $\omega_1 =a_1\omega, \omega_2=a_2\omega$.

In polar coordinates, the general solution  of \eqref{helm1} is given, for $i=1,2$ by
\begin{equation}\label{exp1}
\psi_i(r,\theta)=a_{i,0} J_0(\omega_i r) +\sum_{k=1}^\infty J_k(\omega_i r) [a_{i,k} \cos k\theta + b_{i,k} \sin k\theta].
\end{equation}
 It remains to express the Dirichlet boundary conditions $u_1=u_2=0$ for $r=1$. Using the expression of the derivatives in polar coordinates,
 this leads to the system
 $$\left\lbrace
 \begin{array}{c}
 \cos\theta [\frac{\partial \psi_1}{\partial r} + \frac{\partial \psi_2}{\partial \theta}] - \sin\theta   [\frac{\partial \psi_1}{\partial \theta} - \frac{\partial \psi_2}{\partial r}]=0\\
 \sin\theta  [\frac{\partial \psi_1}{\partial r} + \frac{\partial \psi_2}{\partial \theta}] + \cos\theta [\frac{\partial \psi_1}{\partial \theta} - \frac{\partial \psi_2}{\partial r}]=0
 \end{array}
 \right.$$
 for which we infer 
 \begin{equation}\label{exp2}
 \frac{\partial \psi_1}{\partial r} + \frac{\partial \psi_2}{\partial \theta}=0 ,\qquad \frac{\partial \psi_1}{\partial \theta} - \frac{\partial \psi_2}{\partial r}=0
 \end{equation}
 these equalities being true for $r=1$. Using the expression of $\psi_1,\psi_2$ given in \eqref{exp1}, we get by identification for the constant term
 (and using the fact that $J_0'=-J_1$):
 $$a_{1,0}J_1(a_1\omega)=0,\quad a_{2,0} J_1(a_2\omega)=0.$$
 This provides the sequence of eigenvalues $(\lambda+2\mu) j_{1,k}^2$ and $\mu j_{1,k}^2$ where $j_{1,k}$
 is the sequence of zeros of the Bessel function $J_1$. Among all these values, the smallest one is $\mu j_{1,1}^2$ since
 $\lambda+2\mu > \mu$ by assumption. Therefore, 
 \begin{equation}\label{candidat1}
 \mbox{a candidate to be the first eigenvalue $\Lambda$ is }  \mu j_{1,1}^2.
 \end{equation}
 Now we look at the coefficients in $\cos k\theta$ and $\sin k\theta$. coming from \eqref{exp2}. We obtain the two systems
\begin{equation}\label{sys1}
 \left\lbrace
 \begin{array}{l}
 \omega_1 J'_k(\omega_1) a_{k,1} + k J_k(\omega_2) b_{k,2} =0 \\
 k J_k(\omega_1) a_{k,1} + \omega_2 J'_k(\omega_2) b_{k,2} +  =0
 \end{array}
 \right.
 \end{equation} 
 and
 \begin{equation}\label{sys2}
 \left\lbrace
 \begin{array}{l}
 \omega_1 J'_k(\omega_1) b_{k,1} - k J_k(\omega_2) a_{k,2} =0 \\
 k J_k(\omega_1) b_{k,1} - \omega_2 J'_k(\omega_2) a_{k,2} +  =0 .
 \end{array}
 \right.
 \end{equation} 
 The determinant of these two systems is the same and it must vanish if we look for a non-trivial solution. This leads to the following
 transcendental equations that determines the other eigenvalues
 \begin{equation}\label{trans1}
 a_1a_2\omega^2 J'_k(a_1\omega)J'_k(a_2\omega)-k^2 J_k(a_1\omega)J_k(a_2\omega) =0.
 \end{equation}
 Using the classical relations for the derivative of Bessel functions, we can rewrite \eqref{trans1}
  \begin{equation}\label{trans2}
 \frac{k}{a_1\omega} J_k(a_1\omega)J_{k-1}(a_2\omega) +  \frac{k}{a_2\omega} J_{k-1}(a_1\omega)J_{k}(a_2\omega)
 - J_{k-1}(a_1\omega)J_{k-1}(a_2\omega)=0
 \end{equation}
 or
   \begin{equation}\label{trans3}
 \frac{k}{a_1\omega} J_k(a_1\omega)J_{k+1}(a_2\omega) +  \frac{k}{a_2\omega} J_{k+1}(a_1\omega)J_{k}(a_2\omega)
 - J_{k+1}(a_1\omega)J_{k+1}(a_2\omega)=0 .
 \end{equation}
 Now to determine the first eigenvalue of the elasticity operator, we need to know whether the smallest solution of the previous
 transcendental equations can be smaller than the value $\mu j_{1,1}^2$ already obtained. In that case, the first eigenvalue would be double,
 systems \eqref{sys1} and \eqref{sys2} providing two independent solutions associated to the same eigenvalue.
 
 Let us state the following characterization where we see that the first eigenvalue actually depends on the Poisson coefficient $\nu$:
 \begin{theorem}\label{theopoisson}
 Let $\nu^*$ be the number
\begin{equation}\label{nustar}
  \nu^*:=\frac{j_{1,1}^2-2{j'_{1,1}}^2}{2j_{1,1}^2-2{j'_{1,1}}^2}\simeq 0.349895
  \end{equation}  
where  $j_{1,1}$ is the first zero of the Bessel function $J_1$ and  $j'_{1,1}$ is the first zero of its derivative $J'_1$.
 Assume that the Poisson coefficient $\nu$ satisfies
 \begin{equation}\label{poisson1}
 \nu\leq \nu^*,
 \end{equation}
then the first eigenvalue is given as a solution of the transcendental equation \eqref{trans2}
 for some $k$ and then it is at least double.
Assume that the Poisson coefficient $\nu$ satisfies
 \begin{equation}\label{poisson2}
 \nu\geq \nu^*
 \end{equation}
then the first eigenvalue is $\Lambda=\mu j_{1,1}^2$. Moreover, it is a simple eigenvalue as soon as $\nu > \nu^*$.
 \end{theorem}
 \begin{remark}
 Note that when $\nu=\nu^*$ the first eigenvalue is (at least) triple and equal to $\mu j_{1,1}^2$.
 \end{remark}
 \begin{proof}
 Let us introduce the function $F_k$ defined by
 $$F_k(\omega)= \frac{k}{a_1\omega} J_k(a_1\omega)J_{k-1}(a_2\omega) +  \frac{k}{a_2\omega} J_{k-1}(a_1\omega)J_{k}(a_2\omega)
 - J_{k-1}(a_1\omega)J_{k-1}(a_2\omega).$$
 When $\omega$ is small, using the Taylor expansion of the Bessel function near 0, we obtain
 $$F_k(\omega)=\frac{a_1^{k-1}a_2^{k-1}(a_1^2+a_2^2)}{[(k-1)!]^2 2^{2k+1}k(k+1)}\,\omega^2 + o(\omega^2)$$
 that shows in particular that $F_k(\omega)>0$ for $\omega>0$ small.
 
 Now let us look at $F_1$ and evaluate it at $\omega^*=\sqrt{\mu j_{1,1}^2}=j_{1,1}/a_2$. Since $J_1(a_2\omega^*)=0$ we get
 $$F_1(\omega^*)=\frac{J_0(a_2\omega^*)}{a_1\omega^*}\left(J_1(a_1\omega^*)-a_1\omega^* J_0(a_1\omega^*)\right).$$
 If we can prove that $F_1(\omega^*) \leq 0$, then $F_1$ changes its sign between 0 and $\omega^*$ that implies the fact that
 the first eigenvalue is a zero of the transcendental equation.
 
 Now $J_1(x)-xJ_0(x)=-xJ'_1(x)$ and this is negative between 0 and $j'_{1,1}$ and positive between $j'_{1,1}$ and $j'_{1,2}$.
 On the other hand, the term $J_0(a_2\omega^*)=J_0(j_{1,1} )<0$ therefore, we want to find the case where $x=a_1\omega^*$
 belongs to the interval $[j'_{1,1},j'_{1,2}]$. Now
 $$a_1\omega^*=\sqrt{\frac{\mu}{\lambda+2\mu}} j_{1,1}\in [j'_{1,1},j'_{1,2}] \Leftrightarrow \frac{{j'_{1,1}}^2}{j_{1,1}^2}\leq \frac{1-2\nu}{2-2\nu} \leq
 \frac{{j'_{1,2}}^2}{j_{1,1}^2}$$
 where we use the expression of $\mu/(\lambda+2\mu)$ in term of $\nu$. Solving the previous inequality in $\nu$
 provides the desired result from the left inequality. The right inequality is automatically satisfied since $-1\leq \nu <0.5$.
 
 Now, it remains to prove that, when $\nu \geq \nu^*$ the first eigenvalue is $\mu j_{1,1}^2$ (and is simple when $\nu > \nu^*$).
 Let us introduce $\psi_k(x):=x J_k'(x)/J_k(x)$. It is known that the function $\psi_k$ is decreasing on all interval in $\R_+$
 where it is defined, and in particular on $[0, j_{k,1})$. We refer to \cite{spigler} or \cite{landau99} for that assertion. Moreover $\psi_k(0)=k$.
 This implies that $\psi_k(x) < k$ for $x\in (0,j_{k,1})$ and for any $k$.
 Now, let us assume that $\omega$ is such that $\omega < \sqrt{\mu j_{1,1}^2}$. Since $\lambda+2\mu > \mu$, we have $a_1<a_2$
 and therefore
 $$a_1 \omega < a_2 \omega < a_2  \sqrt{\mu j_{1,1}^2} = j_{1,1} \leq j_{k,1} \quad \mbox{for all } k\geq 1 .$$
 Therefore $J_k(a_1 \omega) >0$ and $J_k(a_2\omega) >0$ for all $\omega < \sqrt{\mu j_{1,1}^2}$. 
 Now, let us rewrite the transcendental equation \eqref{trans1} as (we can divide by $J_k(a_1 \omega) J_k(a_2 \omega) $ that is positive)
 \begin{equation}\label{eqpsik}
 \psi_k(a_1 \omega) \psi_k(a_2 \omega) - k^2 =0 .
 \end{equation}
 Now, the properties we recalled on $\psi_k$ and the fact that 
$a_i \omega < j_{k,1}$ show that the first member of \eqref{eqpsik} is strictly negative when $\omega < \sqrt{\mu j_{1,1}^2}$.
This proves the thesis.
 \end{proof}

 \subsection{Optimality of the disk: first order arguments}
 We wonder whether a Faber-Krahn type inequality holds for the elasticity operator. We will see that it depends actually of the Poisson coefficient.
 Roughly speaking, when the first eigenvalue $\Lambda$ is double, we can prove that the disk is not a minimizer, while when $\Lambda$
 is simple, we can prove that the disk is at least a local minimizer. Let us start by the first possibility:
 \begin{theorem}\label{theonotdisk}
 Assume that the Poisson coefficient $\nu$ satisfies \eqref{poisson1} with a strict inequality. Then the disk does not minimize $\Lambda$ 
 among open sets  of given volume.
 \end{theorem}
 \begin{proof}
 We will use a first order optimality argument for which we need the expression of the eigenvectors. As we have seen in Theorem \ref{theopoisson},
 when $\nu$ satisfies \eqref{poisson1}, the eigenvalue is (at least) double and the two eigenvectors can be obtained through the systems
 \eqref{sys1} and \eqref{sys2} with $\omega$ defined as the smallest solution of all the equations \eqref{trans1} (or  \eqref{trans2},  \eqref{trans3}). The value of the integer  $k$ will not be really important here.
 
 Let us choose for example
 $$a_{1,k}=k J_k(\omega_2), \qquad b_{2,k}=-\omega_1 J'_k(\omega_1)$$
 that satisfy system \eqref{sys1}. (we recall that $\omega_1=a_1\omega$ and $\omega_2=a_2\omega$). Then
 $$\psi_1(r,\theta)=a_{1,k} J_k(\omega_1 r)\cos k\theta,\qquad \psi_2(r,\theta)=b_{2,k} J_k(\omega_2 r)\sin k\theta.$$
 Using \eqref{helm0}, we obtain $u=(u_1,u_2)$ with
 \begin{eqnarray*}
 u_1=a_{1,k}\left(\omega_1 \cos\theta \cos k\theta J'_k(\omega_1 r) + \frac{k\sin\theta}{r} \sin k\theta J_k(\omega_1 r)\right) +\\
 b_{2,k}\left(\omega_2 \sin\theta \sin k\theta J'_k(\omega_2 r) + \frac{k\cos\theta}{r} \cos k\theta J_k(\omega_2 r)\right)
 \end{eqnarray*}
  \begin{eqnarray*}
 u_2=a_{1,k}\left(\omega_1 \sin\theta \cos k\theta J'_k(\omega_1 r) - \frac{k\cos\theta}{r} \sin k\theta J_k(\omega_1 r)\right) +\\
 b_{2,k}\left(-\omega_2 \cos\theta \sin k\theta J'_k(\omega_2 r) + \frac{k\sin\theta}{r} \cos k\theta J_k(\omega_2 r)\right) .
 \end{eqnarray*}
 In principle, we must multiply the previous expressions by a normalization factor in order to satisfy $\int_\Omega u_1^2+u_2^2 =1$,
 but it turns out that this factor has no importance in the computation we present now.
 
 The shape derivative of a multiple eigenvalue is now a classical topic: in the case of the elasticity operator, we refer for example
 to the recent paper \cite{CDM21}. To sum up, let us assume that the eigenvalue has multiplicity $m$ and denote by 
 $u^1,u^2, \ldots u^m$ a set of orthonormal eigenvectors. Then
 if we perturb the boundary of $\Omega$ by a vector field $V$, the first eigenvalue has a semi-derivative (or directional derivative)
 that is given as the smallest eigenvalue of the $m\times m$ matrix $\mathcal{M}$ whose entries are
 $$\mathcal{M}_{i,j}=-\int_{\partial\Omega} \left\lbrack \mu [\nabla u^i:\nabla u^j] + (\lambda+\mu) \operatorname{div} u^i \operatorname{div} u^j\right\rbrack V\cdot n$$
 (where $n$ is here the exterior normal vector). So our thesis will be proved if we can prove that this matrix has a negative eigenvalue
 for a vector field preserving the area, i.e. a vector field $V$ such that $\int_{\partial\Omega} V\cdot n =0$.
 For that purpose, it is sufficient to look at the first term $\mathcal{M}_{1,1}$ and prove that it can be chosen negative (that will imply that
 the symmetric matrix $\mathcal{M}$ is not positive and therefore has a negative eigenvalue). This term being given by
 $$\mathcal{M}_{1,1}=-\int_{\partial\mathbb{D}} \left\lbrack \mu (|\nabla u_1|^2+|\nabla u_2|^2) + (\lambda+\mu) (\operatorname{div} u)^2\right\rbrack V\cdot n$$
 we have to compute on the unit circle $|\nabla u_1|^2, |\nabla u_2|^2$ and $(\operatorname{div} u)^2$.
 
 From the Helmholtz decomposition \eqref{helm0}, it comes 
 $$\operatorname{div} u=\Delta \psi_1=-\omega_1^2 \psi_1=-a_{1,k} \omega_1^2J_k(\omega_1 r) \cos k\theta$$
 so, on the unit circle
 \begin{equation}\label{div1}
 (\operatorname{div} u)^2=a_{1,k}^2 \omega_1^4J_k(\omega_1 )^2 \cos^2 k\theta.
 \end{equation}
 Now, $u_1$ and $u_2$ being constant on the unit circle, we have $|\nabla u_i|^2=\left(\frac{\partial u_i}{\partial r}\right)^2$ with $r=1$.
 Using the formula of $u_1,u_2$, we can write
 \begin{equation}\label{gradu1}
 \frac{\partial u_1}{\partial r}=A_1 \cos\theta \cos k\theta + B_1 \sin\theta \sin k\theta
 \end{equation}
 with
  \begin{eqnarray*}
 A_1=a_{1,k} \omega_1^2 {J_k}''(\omega_1) -k b_{2,k} J_k(\omega_2)+b_{2,k} k \omega_2 J'_k(\omega_2) \\
 B_1=-k a_{1,k} {J_k}(\omega_1) +a_{1,k} k \omega_1 J'_k(\omega_1) +b_{2,k} \omega_2^2 {J_k}''(\omega_2).
 \end{eqnarray*}
 Using the Bessel differential equation to replace ${J_k}''(\omega_i), i=1,2$ by a combination of $J'_k(\omega_i)$ and $J_k(\omega_i)$,
 together with the choice we have done for $a_{1,k}$ and $b_{2,k}$ and the transcendental equation \eqref{trans1},
 we can simplify the previous expressions as 
 \begin{equation}\label{gradu12} 
 A_1=-k \omega_1^2 J_k(\omega_1) J_k(\omega_2), \quad \; B_1=\omega_1\omega_2^2 J'_k(\omega_1) J_k(\omega_2).
 \end{equation}
 In the same way, we obtain 
  \begin{equation}\label{gradu2}
 \frac{\partial u_2}{\partial r}=A_2 \sin\theta \cos k\theta + B_2 \cos\theta \sin k\theta
 \end{equation}
 with 
  \begin{equation}\label{gradu123} 
 A_2=-k \omega_1^2 J_k(\omega_1) J_k(\omega_2)=A_1, \quad \; B_2=-\omega_1\omega_2^2 J'_k(\omega_1) J_k(\omega_2)=-B_1.
 \end{equation}
 Therefore
 \begin{eqnarray*}
 |\nabla u_1|^2+|\nabla u_2|^2=A_1^2\cos^2 k\theta +B_1^2 \sin^2 k\theta = \\
 \omega_1^2 J_k^2(\omega_2) \left(\omega_1^2 k^2 J_k^2(\omega_1) \cos^2 k\theta +\omega_2^4 {J'_k}^2(\omega_1)\sin^2 k\theta\right).
 \end{eqnarray*}
 With $(\operatorname{div} u)^2$ given by \eqref{div1} we finally get
 \begin{equation}\label{m11}
 \mathcal{M}_{1,1}=\omega_1^2 J_k^2(\omega_2) \int_0^{2\pi} \left((\lambda+2\mu)k^2\omega_1^2J_k^2(\omega_1)\cos^2 k\theta
 +\mu \omega_2^4  {J'_k}^2(\omega_1)\sin^2 k\theta\right) V\cdot n .
 \end{equation}
 As explained before,  in order to conclude the proof, it suffices to find a deformation field $V$ such that $\int_0^{2\pi} V\cdot n =0$ and 
 $ \mathcal{M}_{1,1} <0$. Let us choose $V$ such that $V(1,\theta)=\alpha \cos (2k\theta)$. Plugging this value in $ \mathcal{M}_{1,1}$ yields
 \begin{equation}\label{m112}
 \mathcal{M}_{1,1}=\omega_1^2 J_k^2(\omega_2) \frac{\pi \alpha}{2} \,\left((\lambda+2\mu)k^2\omega_1^2J_k^2(\omega_1)
 -\mu \omega_2^4  {J'_k}^2(\omega_1)\right).
 \end{equation}
 The quantity $ \mathcal{M}_{1,1}$ being linear in $\alpha$, in order to conclude we just need to prove that the right-hand side of \eqref{m112}
 cannot be zero. According to Theorem \ref{theopoisson} we know that the eigenvalue satisfies $\Lambda < \mu j_{1,1}^2$, therefore
 $\omega_2=\sqrt{\frac{\Lambda}{\mu}}<j_{1,1}$ and then $J_k(\omega_2)>0$ (for $k\geq 1$, the first zero of $J_k$  is always
 greater or equal to $j_{1,1}$). It remains to consider the quantity $(\lambda+2\mu)k^2\omega_1^2J_k^2(\omega_1)
 -\mu \omega_2^4  {J'_k}^2(\omega_1)$. Using the expression of $\omega_1,\omega_2$ and up to the factor $\omega^2$, it is equal to
 $$Q=k^2 J_k^2(a_1\omega)-\frac{\omega^2}{\mu} {J'_k}^2(a_1\omega).$$
 Since $a_1<a_2$, we know that $J_k(a_1 \omega)>0$. Therefore, $Q=0$ means
 \begin{equation}\label{derj1}
 kJ_k(a_1\omega)-a_2\omega J'_k(a_1\omega) = 0 \quad \mbox{or} \quad kJ_k(a_1\omega)+a_2\omega J'_k(a_1\omega) = 0 .
 \end{equation}
 Let us analyze the first case. From the transcendental equation, we see that
 $$kJ_k(a_1\omega)=a_2\omega J'_k(a_1\omega)  \Rightarrow kJ_k(a_2\omega)=a_1\omega J'_k(a_2\omega) .$$
 Rewriting this in term of the function $\psi_k$ already introduced, this means
 $$\psi_k(a_2 \omega) = \frac{k a_2}{a_1}$$
 but since $ka_2/a_1 > k$ and $\psi_k(x) \leq k$ in this range we see that it is impossible.
 
 Now in the other case,  in the same way thanks to the transcendental equation, we get
 \begin{equation}\label{psiw}
 \psi_k(a_2 \omega) =- \frac{k a_2}{a_1}.
 \end{equation}
 When $k\geq 3$ this is impossible since then $J_k(a_2\omega)$ and $J'_k(a_2\omega)$ are both positive (we recall that
 we are in the case where $\omega \leq \sqrt{\mu}j_{1,1} \Rightarrow a_2 \omega \leq j_{1,1}$). It remains the case $k=2$.
 In that case, due to the fact that $\psi_2$ is decreasing, we infer $\psi_2(a_2 \omega)\geq \psi_2(j_{1,1})$.
 But since $j_{1,1} J'_2(j_{1,1}) = -2 J_2(j_{1,1} )$ this would imply with \eqref{psiw}
 $$-\frac{2 a_2}{a_1} \geq -2 \Rightarrow a_2 \leq a_1$$
 a contradiction since we know that $a_2> a_1$.
 
This finishes the proof
 of non optimality of the disk in that case.
 \end{proof}
  \subsection{Optimality of the disk: second order arguments}\label{seconddisk}
 Let us assume that $\Omega$ is the unit disk $\Omega=\mathbb{B}_2$ and assume that $\Lambda$ is simple.
 We know, according to Theorem \ref{theopoisson} that it is the case as soon as $\nu > \nu^*$ and moreover  $\Lambda(\Omega)=\mu j_{1,1}^2$. 
 We also know, from the proof of Theorem \ref{theopoisson} that the associated eigenspace is spanned by the normalized vector $U=[u_1,u_2]^\top$, reading in polar coordinates $(r,\theta)$ as
 $$
 u_1=-\alpha \sin\theta J_1(j_{1,1}r)\quad \text{and}\quad  u_2=\alpha \cos\theta J_1(j_{1,1}r)
 $$
where $\alpha=\frac{1}{\sqrt \pi |J_0(j_{1,1})|}$.  
Our aim is to prove that in that case, the first order shape derivative of the functional $\mathcal{F}(\Omega):=|\Omega| \Lambda(\Omega)$
is zero (for any vector field $V$) while the second order shape derivative is a positive quadratic form.

A consequence of the general formulae for the second shape derivative given in Proposition~\ref{Pr:DifferentiabilityFormulae} is:
\begin{proposition}\label{cor:LambSecBall}
Assume that $\Omega=\mathbb{B}_2$ is the unit disk in $\R^2$. 
 Assume that the Poisson coefficient $\nu$ satisfies $ \nu> \nu^*$.
Then, the second order derivative of $\Lambda$ at $\Omega=\mathbb{B}_2$ reads
 \begin{eqnarray}\label{Eq:D2Lambda}
   \langle d^2\Lambda(\Omega)V,V\rangle &=& -\mu\int_{\partial \Omega}\frac{\partial^2  u_\Omega }{\partial n^2}\cdot \frac{\partial  u_\Omega }{\partial n}(V\cdot n)^2 \nonumber  \\ 
&& -2\Lambda \int_\Omega |\dot{u}_V|^2+2\mu \int_\Omega |\nabla \dot{u}_V|^2+2(\lambda+\mu)\int_\Omega (\operatorname{div}\dot{u}_V)^2.
\end{eqnarray} 
\end{proposition}

\begin{proof}
This follows by observing that, in such a case, 
\begin{itemize}
\item one has $\operatorname{div}u_\Omega=0$ in $\Omega$ ; 
\item furthermore, using the standard decomposition of the Laplacian on $\partial\Omega$ yields
$$
0=-\Lambda(\Omega)u_\Omega-(\lambda+\mu)\nabla \operatorname{div}u_\Omega=\mu \Delta u_\Omega=\frac{\partial^2  u_\Omega }{\partial n^2}+H  \frac{\partial  u_\Omega}{\partial n}+\Delta_{\partial\Omega} u_\Omega,
$$
where $\Delta_{\partial\Omega}$ stands for the Laplace-Beltrami tangential operator on $\partial\Omega$. It follows that $\Delta_{\partial\Omega} u_\Omega=0$ on $\partial\Omega$, and thus,
$$
H \left| \frac{\partial  u_\Omega}{\partial n}\right|^2+2\frac{\partial^2  u_\Omega }{\partial n^2}\cdot \frac{\partial  u_\Omega }{\partial n}=\frac{\partial^2  u_\Omega }{\partial n^2}\cdot \frac{\partial  u_\Omega }{\partial n}\quad \text{on }\partial\Omega .
$$
\end{itemize}
\end{proof}
Our first task is to compute explicitly the second derivative of $\Lambda$.
According to Proposition~\ref{cor:LambSecBall}, one has
 $$
    \langle d^2\Lambda(\Omega)V,V\rangle = -\mu A_1-2\Lambda \int_\Omega |\dot{u}_V|^2+ A_2+2(\lambda+\mu)A_3,
 $$
 where
 \begin{eqnarray*}
 A_1 &=& \int_{\partial \Omega}\frac{\partial^2  u_\Omega }{\partial n^2}\cdot \frac{\partial  u_\Omega }{\partial n}(V\cdot n)^2\\
 A_2 &=& 2\mu\int_\Omega |\nabla \dot{u}_V|^2\\
 A_3 &=& \int_\Omega (\operatorname{div}\dot{u}_V)^2.
 \end{eqnarray*}
 Let us compute each term separately. To this aim, it is convenient to denote by $\varphi$ the function $V\cdot n$, defined on the boundary of $\mathbb{B}_2$, expanding in the $\theta$-coordinate as
 \begin{equation}\label{expand:varphiFourier}
 \varphi(\theta)=\sum_{k=0}^{+\infty} \alpha_k\cos (k\theta)+\beta_k\sin (k\theta).
 \end{equation}
 \paragraph{A simplified expression of $    \langle d^2\Lambda(\Omega)V,V\rangle $.} 
 According to the Green formula and Proposition~\ref{Pb:der}, one has
  \begin{eqnarray*}
 A_2&=&2\mu \int_{\partial\Omega}\dot{u}_V\cdot \frac{\partial \dot{u}_V}{\partial n}-2\mu\int_{\Omega}\dot{u}_V\cdot \Delta \dot{u}_V\\
 &=& 2\mu\int_{\partial\Omega}\dot{u}_V\cdot \frac{\partial \dot{u}_V}{\partial n}+2\Lambda(\Omega)\int_\Omega |\dot{u}_V|^2+2(\lambda+\mu)\int_\Omega \dot{u}_V\cdot \nabla \operatorname{div}\dot{u}_V .
  \end{eqnarray*}
 Let us now use the equation satisfied by $\dot{u}_V$. We obtain
  \begin{eqnarray*}
 A_2 &=& 2\mu\int_{\partial\Omega}\dot{u}_V\cdot \frac{\partial \dot{u}_V}{\partial n}+\Lambda(\Omega)\int_\Omega |\dot{u}_V|^2-2(\lambda+\mu)\int_\Omega ( \operatorname{div}\dot{u}_V)^2+2(\lambda+\mu)\int_{\partial \Omega}\operatorname{div}\dot{u}_V(\dot{u}_V\cdot n) .
  \end{eqnarray*}
Note that 
\begin{eqnarray*}
\dot{u}_V\cdot n&=&-\varphi \nabla u_\Omega n \cdot n  = -\varphi \sum_{i,j}\frac{\partial u_{\Omega,i}}{\partial x_j}n_jn_i= -\varphi \sum_i \frac{\partial u_{\Omega,i}}{\partial n}n_i=-\varphi \operatorname{div}u_\Omega=0\quad \text{on }\partial\Omega.
\end{eqnarray*}
We get
   \begin{eqnarray}\label{train0954}
 A_2 &=& 2\mu\int_{\partial\Omega}\dot{u}_V\cdot \frac{\partial \dot{u}_V}{\partial n}+2\Lambda(\Omega)\int_\Omega |\dot{u}_V|^2-2(\lambda+\mu)\int_\Omega ( \operatorname{div}\dot{u}_V)^2 .
  \end{eqnarray}
As a consequence,
$$
    \langle d^2\Lambda(\Omega)V,V\rangle = -\mu A_1+2\mu\int_{\partial\Omega}\dot{u}_V\cdot \frac{\partial \dot{u}_V}{\partial n}.
 $$
Let us now expand this expression into a sum of squares.
 
 \paragraph{Computation of $A_1$.} One has
 $$
 \frac{\partial^2  u_\Omega }{\partial n^2}\cdot \frac{\partial  u_\Omega }{\partial n}=\alpha^2j_{1,1}^3J_1'(j_{1,1})J_1''(j_{1,1})=-\alpha^2j_{1,1}^2J_0(j_{1,1})^2,
 $$
 by noting that $J_1'(j_{1,1})=J_0(j_{1,1})$ and $j_{1,1}^2J_1''(j_{1,1})=-j_{1,1}J_1'(j_{1,1})$. It follows that
 $$
 A_1=\mu\alpha^2 j_{1,1}^2J_0(j_{1,1})^2\int_{\partial\Omega}(V\cdot n)^2 = 
\Lambda\left(2\alpha_0^2+\sum_{k=1}^{+\infty}(\alpha_k^2+\beta_k^2)\right) .
 $$

 \paragraph{Computation of $\boldsymbol{I:=\int_{\partial\Omega}\dot{u}_V\cdot \frac{\partial \dot{u}_V}{\partial n}}$.} 
Recall that $\dot{u}_V$ satisfies
\begin{equation}\label{dotuBoundary}
\dot{u}_V=\alpha j_{1,1}J_0(j_{1,1})\varphi(\theta)[\sin\theta,-\cos \theta]^\top\quad \text{on }\partial\Omega.
\end{equation}
To compute $\dot{u}_V$ inside the domain $\Omega$, we will use Lemma ~\ref{lem:helm} with $v=\dot{u}_V$ and $f=d\Lambda(\Omega,V) u_\Omega$.
According to formulae \eqref{Pr:DifferentiabilityFormulaeBis} and since $\mathrm{div} u_\Omega=0$, we finally obtain for the unit disk
\begin{equation}\label{firstderdisk}
d\Lambda(\Omega,V)=-2\mu j_{1,1}^2 \alpha_0= -2\Lambda(\Omega) \alpha_0.
\end{equation}
Since the first derivative of the area is $dA(\Omega,V)=\int_{\partial \Omega} V\cdot n=2\pi \alpha_0$, we recover the fact that
the first derivative of the functional $\mathcal{F}$ is zero at the disk (in other terms, the disk is a critical point).

We now use the decomposition of Lemma~\ref{lem:helm} on the unit circle taking profit that $u_\Omega$ vanishes on the boundary.
Therefore, by writing $\psi_1$ and $\psi_2$ in the polar coordinates $(r,\theta)$, one has on the boundary
$$
\left\{\begin{array}{l}
\dot{u}_{V,1}=\partial_x\psi_1-\partial_y\psi_2=\cos\theta \left(\partial_r\psi_1-\frac{1}{r}\partial_\theta\psi_2\right)-\sin\theta \left(\frac{1}{r}\partial_\theta\psi_1+\partial_r\psi_2\right)\\
\dot{u}_{V,2}=\partial_y\psi_1+\partial_x\psi_2=\sin\theta \left(\partial_r\psi_1-\frac{1}{r}\partial_\theta\psi_2\right)+\cos\theta \left(\frac{1}{r}\partial_\theta\psi_1+\partial_r\psi_2\right) .
\end{array}
\right.
$$
By using \eqref{dotuBoundary}, we infer that
\begin{equation}\label{eq0953}
\partial_r\psi_1-\partial_\theta\psi_2=0\quad \text{and}\quad \partial_\theta\psi_1+\partial_r\psi_2=-\frac{j_{1,1}}{\sqrt \pi}\varphi(\theta).
\end{equation}
For the sake of notational clarity, let us introduce 
$$
\omega=\sqrt{\frac{\mu}{\lambda+2\mu}}.
$$
According to Lemma~\ref{lem:helm}, the functions $\psi_1$ and $\psi_2$ solve the PDEs
$$
-\Delta \psi_1=\omega^2 j_{1,1}^2\psi_1\quad \text{and}\quad -\Delta \psi_2=j_{1,1}^2\psi_2\quad \text{in }\Omega.
$$
We infer that $\psi_1$ and $\psi_2$ expand as
\begin{eqnarray*}
\psi_1 &=& a_{1,0}J_0(\omega j_{1,1}r)+\sum_{k=1}^{+\infty}\left(a_{1,k}\cos (k\theta)+b_{1,k}\sin(k\theta)
\right)J_k(\omega j_{1,1}r)\\
\psi_2 &=& a_{2,0}J_0( j_{1,1}r)+\sum_{k=1}^{+\infty}\left(a_{2,k}\cos (k\theta)+b_{2,k}\sin(k\theta)
\right)J_k( j_{1,1}r) .
\end{eqnarray*}
Plugging these expressions into \eqref{eq0953} allows us to compute the Fourier coefficients characterizing $\psi_1$ and $\psi_2$:
$$
\left\{\begin{array}{l}
a_{1,0}=0, \quad b_{1,0}\text{ is arbitrary}\\
a_{1,k}j_{1,1}\omega J_k'(j_{1,1}\omega)-jJ_1(j_{1,1})b_{2,k}=0, \qquad k\geq 1\\
-ka_{1,k}J_k(j_{1,1}\omega)+j_{1,1}J_k'(j_{1,1})b_{2,k}=-\frac{j_{1,1}}{\sqrt \pi}\beta_k\\
kb_{1,k}J_j(j_{1,1}\omega)+j_{1,1}J_k'(j_{1,1})a_{2,k}=-\frac{j_{1,1}}{\sqrt \pi}\alpha_k\\
b_{1,k}j_{1,1}\omega J_k'(j_{1,1}\omega)+kJ_k(j_{1,1})a_{2,k}=0
\end{array}\right.
$$
which, after easy computations,  reduces into
\begin{eqnarray*}
a_{1,k} & =& -\frac{kj_{1,1}J_k(j_{1,1})}{\sqrt \pi (j_{1,1}^2\omega J_{k}'(\omega j_{1,1})J_k'(j_{1,1})-k^2J_k(j_{1,1})J_k(\omega j_{1,1}))}\beta_k\\
a_{2,k} & =& \frac{kj_{1,1}J_k'(j_{1,1})}{\sqrt \pi (j_{1,1}^2\omega J_{k}'(\omega j_{1,1})J_k'(j_{1,1})-k^2J_k(j_{1,1})J_k(\omega j_{1,1}))}\alpha_k\\
b_{1,k} & =& -\frac{\omega j_{1,1}^2J_k(j_{1,1}\omega)}{\sqrt \pi (j_{1,1}^2\omega J_{k}'(\omega j_{1,1})J_k'(j_{1,1})-k^2J_k(j_{1,1})J_k(\omega j_{1,1}))}\alpha_k\\
b_{2,k} & =& -\frac{\omega j_{1,1}^2J_k'(j_{1,1}\omega)}{\sqrt \pi (j_{1,1}^2\omega J_{k}'(\omega j_{1,1})J_k'(j_{1,1})-k^2J_k(j_{1,1})J_k(\omega j_{1,1}))}\beta_k .\\
\end{eqnarray*}
We know the explicit expression of $\psi_1$ and $\psi_2$. 
We are now in position to compute $I$. One has
\begin{eqnarray*}
I &=& \int_0^{2\pi}\left(\dot{u}_{V,1}\frac{\partial \dot{u}_{V,1}}{\partial r}+\dot{u}_{V,2}\frac{\partial \dot{u}_{V,2}}{\partial r}\right)\, d\theta \\
&=& \frac{j_{1,1}}{\sqrt \pi}\int_0^{2\pi}\left(\frac{\partial\psi_1}{\partial\theta}-\left(\frac{\partial^2\psi_1}{\partial r\partial\theta}+\frac{\partial^2\psi_2}{\partial r^2}\right)\right)\varphi(\theta)\, d\theta +(2\alpha_0)^2 \int_0^{2\pi} u_\Omega . \frac{\partial u_\Omega}{\partial r}\,d\theta.
\end{eqnarray*}
Note that 
\begin{eqnarray*}
\frac{\partial^2\psi_1}{\partial r\partial\theta}&=& \omega j_{1,1}\sum_{k=1}^{+\infty}kJ_k'(\omega j_{1,1})\left(b_{1,k}\cos (k\theta)-a_{1,k}\sin (k\theta)\right)\\
\frac{\partial^2\psi_2}{\partial r^2}&=& j_{1,1}^2\sum_{k=1}^{+\infty} J_k''(j_{1,1})\left(a_{2,k}\cos (k\theta)+b_{2,k}\sin (k\theta)\right)+ j_{1,1}^2a_{2,0}J_1''( j_{1,1})
\end{eqnarray*}
on $\partial\Omega$. Now, using
$$
 j_{1,1}^2J_k''( j_{1,1})=- j_{1,1}J_k'( j_{1,1})+(k^2- j_{1,1}^2)J_k( j_{1,1}),
$$
it follows from easy, but lengthly computations that
\begin{eqnarray*}
-\frac{j_{1,1}}{\sqrt \pi}\int_0^{2\pi}\left(\frac{\partial^2\psi_1}{\partial r\partial\theta}+\frac{\partial^2\psi_2}{\partial r^2}\right)\varphi(\theta)\, d\theta &=& -\omega j_{1,1}^3\sum_{k=1}^{+\infty}\frac{k^2J_k(j_{1,1})J_k'(\omega j_{1,1})(\alpha_k^2+\beta_k^2)}{j_{1,1}^2\omega J_{k}'(\omega j_{1,1})J_k'(j_{1,1})-k^2J_k(j_{1,1})J_k(\omega j_{1,1})}\\
&&  +\omega j_{1,1}^3\sum_{k=1}^{+\infty}\frac{(-j_{1,1}J_k'(j_{1,1})+(k^2-j_{1,1}^2)J_k(j_{1,1}))J_k'(\omega j_{1,1})(\alpha_k^2+\beta_k^2)}{j_{1,1}^2\omega J_{k}'(\omega j_{1,1})J_k'(j_{1,1})-k^2J_k(j_{1,1})J_k(\omega j_{1,1})}\\
&=& -\omega j_{1,1}^4\sum_{k=1}^{+\infty}\frac{(J_k'(j_{1,1})+j_{1,1}J_k(j_{1,1}))J_k'(\omega j_{1,1})(\alpha_k^2+\beta_k^2)}{j_{1,1}^2\omega J_{k}'(\omega j_{1,1})J_k'(j_{1,1})-k^2J_k(j_{1,1})J_k(\omega j_{1,1})} .
\end{eqnarray*}
Similarly, since 
$$
\frac{\partial\psi_1}{\partial \theta}=\sum_{k=1}^{+\infty}kJ_k(\omega j_{1,1})\left(b_{1,k}\cos (k\theta)-a_{1,k}\sin (k\theta)\right)\qquad \text{on }\partial \Omega,
$$
it follows that
\begin{eqnarray*}
\frac{j_{1,1}}{\sqrt \pi}\int_0^{2\pi} \frac{\partial\psi_1}{\partial\theta}\, d\theta &=& - j_{1,1}^2\sum_{k=1}^{+\infty}\frac{k^2J_k(\omega j_{1,1})J_k(j_{1,1})(\alpha_k^2+\beta_k^2)}{j_{1,1}^2\omega J_{k}'(\omega j_{1,1})J_k'(j_{1,1})-k^2J_k(j_{1,1})J_k(\omega j_{1,1})} .
\end{eqnarray*}
\paragraph{Conclusion.} Finally, we obtain the following expression of $ \langle d^2\Lambda(\Omega)V,V\rangle$ by combining all the results above:
$$
 \langle d^2\Lambda(\Omega)V,V\rangle=\mu j_{1,1}^2\left(6\alpha_0^2+\sum_{k=1}^{+\infty}c_k(\alpha_k^2+\beta_k^2)\right),
$$
where 
$$
c_k=\frac{k^2J_k(\omega j_{1,1})J_k(j_{1,1})-\omega j_{1,1}^2J_k'(j_{1,1})J_k'(\omega j_{1,1})-2k \omega j_{1,1}^2 J_k(j_{1,1})J_k'(\omega j_{1,1})}{j_{1,1}^2\omega J_{k}'(\omega j_{1,1})J_k'(j_{1,1})-k^2J_k(j_{1,1})J_k(\omega j_{1,1})}.
$$
These computations allow us to state:
\begin{theorem}
Let $\mathcal{F}$ the shape functional defined by $\mathcal{F}(\Omega)=|\Omega| \Lambda(\Omega)$ and $\Omega$ be the unit disk.
Then $d\mathcal{F}(\Omega,V)=0$ and
\begin{equation}\label{der2F}
\langle d^2 \mathcal{F}(\Omega),V,V\rangle \geq A_0 \|\hat{V}\|_{H^1(\partial\Omega)}^2,
\end{equation}
where $\hat{V}$ denotes the projection of $V\cdot n$ on the orthogonal space to $span\{1,\cos\theta,\sin\theta\}$.
Therefore the unit disk is a local minimum in a weak sense.
\end{theorem}

\begin{proof}
The fact that the first derivative of $\mathcal{F}$ vanishes at the disk has already been proved. Let us compute the second shape derivative.
Denoting by $A$ the area, we have
$$d^2\mathcal{F}=\Lambda d^2A +2 d\Lambda dA +A d^2\Lambda.$$
Using $dA=\int_{\partial\Omega} \varphi$, $d^2A=\int_{\partial\Omega} H \varphi^2$ where the mean curvature $H$ equals $1$
and $d\Lambda=-2\Lambda \int_{\partial\Omega} \varphi$, we finally get with the above expression of $d^2\Lambda$ the following expansion
for the second derivative
$$
 \langle d^2\mathcal{F}(\Omega)V,V\rangle=\pi \Lambda \sum_{k=1}^{+\infty}C_k(\alpha_k^2+\beta_k^2)
$$
where $\alpha_k$ and $\beta_k$ are the coefficients in the expansion \eqref{expand:varphiFourier} of $\varphi$ and
$$
C_k=2j_{1,1}^2 \omega \frac{kJ'_k(\omega j_{1,1})J_k(j_{1,1})}{k^2J_k(\omega j_{1,1})J_k(j_{1,1})-j_{1,1}^2\omega J'_k(\omega j_{1,1})J'_k(j_{1,1})}
$$
with $\omega=\sqrt{\mu/(\lambda+2\mu)}<1$. 
We remark that no terms come from $k=0$ and $k=1$ ($C_1=0$). This is due to the invariance of the functional $\mathcal{F}$
under dilation and rotation.
We claim that  each $C_k$ is positive for $k\geq 2$.  Indeed, we have already seen in the proof of Theorem \ref{theopoisson}
that the denominator of $C_k$ is positive.  The first term in the numerator is also positive since, for $k\geq 2$, $J_k(j_{1,1}) >0$.
For the second term we need to be more precise. Since we are in the case where the first eigenvalue of the disk is $\mu j_{1,1}^2$
we know that $\nu >\nu^*$. Now,
$$\omega=\sqrt{\frac{\mu}{\lambda+2\mu}}=\sqrt{\frac{1-2\nu}{2-2\nu}} < \sqrt{\frac{1-2\nu^*}{2-2\nu^*}}= \frac{j'_{1,1}}{j_{1,1}} <\frac{1}{2}.$$
Therefore $\omega j_{1,1} <  j_{1,1}/2 < j'_{k,1}$ for any $k\geq 2$ and $J'_k(\omega j_{1,1})>0$.

To conclude the proof, we look at the asymptotic behaviour of $C_k$ for $k$ large. When $x$ is fixed, we have for $k\,$ large
$$J_k(x)\sim \frac{x^k}{2^k k!} -\frac{x^{k+2}}{2^{k+2} (k+1)!} \quad \mbox{ and } \quad J'_k(x)\sim \frac{x^{k-1}}{2^k (k-1)!}-\frac{(k+2)x^{k+1}}{2^{k+2} (k+1)!}  .$$
Then the numerator $N_k$ of $C_k$  satisfies
$$N_k\sim \frac{ j_{1,1}^{2k+1}\omega^k}{2^{2k-1}[(k-1)!]^2}$$
while the denominator $D_k$ of $C_k$ satisfies
$$D_k\sim  \frac{ j_{1,1}^{2k+2}\omega^k}{2^{2k+2}(k-1)!(k+1)!}\left(2(1+\omega^2)-\frac{\omega^2j_{1,1}^2}{k}\right)\quad \text{as }k\to +\infty.$$
Finally
$$C_k\sim \frac{4k(k+1)}{j_{1,1}\left(2(1+\omega^2)-\frac{\omega^2j_{1,1}^2}{k}\right)} \geq k(k+1)\quad \text{as }k\to +\infty.$$
The conclusion follows since the $H^1$ norm of the projection of $\varphi$  on the orthogonal space to $\operatorname{span}\{1,\cos\theta,\sin\theta\}$ is
$$\|\varphi\|_{H^1}^2=\sum_{k+2}^{+\infty} (k^2+1) (\alpha_k^2+\beta_k^2).$$
\end{proof}

  \section{Some particular domains}
The aim of this section is to find (simple) domains which may have a lower first eigenvalue than the disk, at least when $\nu\geq \nu^*$.
For that purpose, we will give first explicit examples for which we can give the exact value of $\Lambda$. Let us mention that these examples
are very similar to the ones found by Kawohl-Sweers in \cite{kawohl-sweers}.
Then we will consider the case of rectangles. In that case, we are not able to give the exact value of $\Lambda$ but we can estimate it from
above with a good precision.
  \subsection{Rhombi}\label{secrhombi}
In this section, we  discuss the following question: does there exist some domain in the plane for which the
 eigenvector is given by twice the same function, i.e. $U(x,y)=(u(x,y),u(x,y))$.
 As we will see, this is possible and we can even, in that case, find an explicit eigenvector and an explicit eigenvalue.
More precisely we will find some parallelograms, actually rhombi, (depending on the Lam\'e coefficients $\lambda,\mu$) fulfilling this condition
 and the associated eigenvalue will be quite simple and only depend on the area of the parallelogram.
 
We work by analysis and synthesis.
 
 \paragraph{Analysis.}
 Let us assume that the domain $\Omega \subset \mathbb{R}^2$ has the property that its eigenvector is given by $U=(u(x,y),u(x,y))$.
 Thus $\mathrm{div}(U)=\frac{\partial u}{\partial x} + \frac{\partial u}{\partial y}$. We replace in the eigenvector equation
 \eqref{pdeLame1} and we make the difference of the two equations to obtain
 \begin{equation}
 \frac{\partial }{\partial x} \mathrm{div} U - \frac{\partial }{\partial y} \mathrm{div} U =0 .
 \end{equation}
 Therefore (locally, but then globally by analyticity), we have
 \begin{equation}\label{transp}
 \mathrm{div} U = \frac{\partial u}{\partial x} + \frac{\partial u}{\partial y} =  f(x+y)
 \end{equation}
  for some analytic function $f$. Solving this transport equation \eqref{transp} provides the existence of two analytic functions
  $\varphi$ and $\psi$ such that finally
\begin{equation}\label{defu}
 u(x,y)= \varphi(x-y) + \psi(x+y).
 \end{equation} 
 Now we come back to the system \eqref{pdeLame1}: we have $\Delta u=2 (\varphi^{\prime\prime}(x-y) +\psi^{\prime\prime}(x+y) )$
 and $ \mathrm{div} U = 2 \psi^\prime (x+y)$. Therefore, using the change of variable $v=x-y, w=x+y$ we see that $\varphi$ and $\psi$
 must satisfy
 $$-2\mu (\varphi^{\prime\prime}(v) +\psi^{\prime\prime}(w) ) - 2(\lambda+\mu) \psi^{\prime\prime}(w) = \Lambda (\varphi(v) + \psi(w)).$$
 In this equation, we can separate variables to get the existence of some constant $C$ such that
 $$-2(\lambda+2\mu) \psi^{\prime\prime}(w) - \Lambda \psi(w) = C = 2\mu (\varphi^{\prime\prime}(v) + \Lambda \varphi(v)).$$
 Solving this equation separately in $\psi$ and $\varphi$ yields
 \begin{equation}\label{eqpsi}
 \psi(w)=A_1\cos \omega_1 w + B_1\sin \omega_1 w  -\frac{C}{\Lambda} \quad\mbox{ with } \omega_1^2=\frac{\Lambda}{2\lambda+4\mu}
 \end{equation}
 and
\begin{equation}\label{eqphi}
 \varphi(v)=-A_2\cos \omega_2 v - B_2\sin \omega_2 v  +\frac{C}{\Lambda} \quad\mbox{ with } \omega_2^2=\frac{\Lambda}{2\mu} .
 \end{equation}
 Adding \eqref{eqphi} and \eqref{eqpsi}, we get by \eqref{defu}
 $$
 u(x,y)=u(v,w)= A_1\cos \omega_1 w + B_1\sin \omega_1 w -A_2\cos \omega_2 v - B_2\sin \omega_2 v
 $$
 that can also be rewritten as 
  \begin{equation}\label{formulau}
 u(v,w)=C_1 \sin(\omega_1 w-\theta_1) - C_2 \sin(\omega_2 v-\theta_2).
 \end{equation}
 With this expression of $u$ we have completely taken into account the eigen-equation. It just remain to express the Dirichlet
 boundary condition. In other words, domains $\Omega$ that will satisfy the property (that the eigenvector is of the kind $(u,u)$) are
 those domains on which a function $u(v,w)$ given by \eqref{formulau} vanishes on the boundary of $\Omega$.

 \paragraph{Synthesis.}
We will prove below that necessarily $C_1=C_2$ in the expression \eqref{formulau}. So let us assume that $C_1=C_2$ and let us 
investigate the set of points where $u$ vanishes. In that case we have to solve
 $\sin(\omega_1 w-\theta_1) = \sin(\omega_2 v-\theta_2)$, therefore, coming back to the variables $x,y$:
 $$u=0 \Leftrightarrow \left\lbrace
 \begin{array}{l}
 \omega_1(x+y)-\omega_2(x-y)=\theta_1-\theta_2 + 2k\pi, \;k\in \mathbb{Z} \\
  \omega_1(x+y)+\omega_2(x-y)=\theta_1+\theta_2 + (2k'+1) \pi, \;k'\in \mathbb{Z} \\
 \end{array}\right.$$
 or, it can also be written using the definition of $\omega_1,\omega_2$ and introducing the real numbers $a_1=\theta_1-\theta_2$
 and $a_2=\theta_1+\theta_2$
\begin{equation}\label{charp}
 \left\lbrace
 \begin{array}{l}
 \left(\frac{1}{\sqrt{\lambda+2\mu}}-\frac{1}{\sqrt{\mu}}\right) x+ \left(\frac{1}{\sqrt{\lambda+2\mu}}+\frac{1}{\sqrt{\mu}}\right) y =
 \sqrt{\frac{2}{\Lambda}}(a_1+ 2k\pi), \;k\in \mathbb{Z} \\
 \left(\frac{1}{\sqrt{\lambda+2\mu}}+\frac{1}{\sqrt{\mu}}\right) x+ \left(\frac{1}{\sqrt{\lambda+2\mu}}-\frac{1}{\sqrt{\mu}}\right) y =
 \sqrt{\frac{2}{\Lambda}}(a_2+ + (2k'+1) \pi), \;k'\in \mathbb{Z} .\\
 \end{array}\right.
 \end{equation} 
 This corresponds to equations of line segments with two specific normal vectors. 
Therefore, the domain $\Omega$ should be a parallelogram delimited by such parallel line segments. But we have to make more precise what
line segments. 
 To simplify the notations, let us introduce
 $$\alpha=\frac{1}{\sqrt{\lambda+2\mu}}-\frac{1}{\sqrt{\mu}} , \quad \beta=\frac{1}{\sqrt{\lambda+2\mu}}+\frac{1}{\sqrt{\mu}}$$
 and the normal vectors
 $$\mathbf{e_1}=\left(\begin{array}{c}
 \alpha \\
 \beta
 \end{array}\right)\qquad  \mathbf{e_2}=\left(\begin{array}{c}
 \beta \\
 \alpha
 \end{array}\right).$$
 Let us assume that the parallelogram is defined by the four equations
 $$\left\lbrace
 \begin{array}{l}
  \mathbf{e_1}\cdot X=\xi_1 \\
  \mathbf{e_1}\cdot X=\hat{\xi}_1
  \end{array} \right. \qquad 
  \left\lbrace
 \begin{array}{l}
  \mathbf{e_2}\cdot X=\xi_2 \\
  \mathbf{e_2}\cdot X=\hat{\xi}_2 .
  \end{array} \right. $$
According to \eqref{charp}, we must have
$$\xi_1= \sqrt{\frac{2}{\Lambda}}(a_1+ 2k\pi) \quad \hat{\xi}_1= \sqrt{\frac{2}{\Lambda}}(a_1+ 2\hat{k}\pi)$$
therefore
$\hat{\xi}_1-\xi_1=  \sqrt{\frac{2}{\Lambda}} 2m\pi$ for some integer $m$ that cannot be zero. Let us take the smallest possible value
$m=1$ (or $m=-1$). This shows that
\begin{equation}
\hat{\xi}_1-\xi_1=  \sqrt{\frac{2}{\Lambda}} 2\pi .
\end{equation}
Exactly in the same way, we get
\begin{equation}\label{hatxi}
\hat{\xi}_2-\xi_2= \sqrt{\frac{2}{\Lambda}} 2\pi .
\end{equation}
In particular we see that the parallelogram must satisfy $\hat{\xi}_1-\xi_1=\hat{\xi}_2-\xi_2$ and therefore, it is a rhombus.
We are going to give a simple relation between the area of the parallelogram and the eigenvalue $\Lambda$.
Assume that the parallelogram has vertices $A,B,C,D$ with $B=A+\rho_1 \mathbf{e_1}^\perp$ and $D=A+\rho_2 \mathbf{e_2}^\perp$
where $\mathbf{e_1}^\perp$ and $\mathbf{e_2}^\perp$ are the vectors respectively orthogonal to $\mathbf{e_1}$
and $\mathbf{e_2}$ with the same norm. The line $(AB)$ corresponds to $\xi_1$ and the line $(AD)$ to $\hat{\xi}_2$ in the previous notations.
Then the length of the basis $AB$ is $AB=\rho_1 \|\mathbf{e_1}^\perp\|$. On the other hand, the height $h$ of the parallelogram is given
by the distance between $B$ and its orthogonal projection $B_1$ on the line $(CD)$. In other words the height is given by
$$h=\frac{1}{\|\mathbf{e_1}\|} BB_1\cdot\mathbf{e}_1.$$
Now $BB_1\cdot \mathbf{e}_1=AB_1\cdot \mathbf{e}_1=\hat{\xi}_1 - \xi_1$ by definition of the two lines. Finally the area of the parallelogram $\Omega$,
that is $AB\times h$ is given by
$$|\Omega|=\rho_1 \|\mathbf{e_1^\perp}\| \frac{1}{\|\mathbf{e_1}\|}( \hat{\xi}_1 - \xi_1 )= \rho_1  \sqrt{\frac{2}{\Lambda}} 2\pi.$$
It remains to express $\rho_1$ taken into account the relation \eqref{hatxi}.
Let $B_2$ be the orthogonal projection of $B$ on the line $(AD)$. By definition we have $BB_2\cdot \mathbf{e_2}=-\hat{\xi}_2+\xi_2$.
Now 
$$\rho_1 \mathbf{e_1^\perp}\cdot \mathbf{e_2}=AB\cdot \mathbf{e_2}=B_2B\cdot \mathbf{e_2}=\hat{\xi}_2-\xi_2 .$$
Thus
$$\rho_1=\frac{(\hat{\xi}_2 - \xi_2)}{\mathbf{e_1}^\perp \cdot\mathbf{e_2}}= 
\sqrt{\frac{2}{\Lambda}} 2\pi \frac{\sqrt{\mu(\lambda+2\mu)}}{4}.$$
Therefore we have proved that the area of the parallelogram is given by
\begin{equation}\label{areaparall}
|\Omega|=\frac{2\pi^2 \sqrt{\mu(\lambda+2\mu)}}{\Lambda}.
\end{equation}
Let us rephrase this formula in stating the following theorem:
\begin{theorem}
Let $\Omega$ be a parallelogram defined by the four lines
$$\left\lbrace
 \begin{array}{l}
  \mathbf{e_1}\cdot X=\xi_1 \\
  \mathbf{e_1}\cdot X=\hat{\xi}_1
  \end{array} \right. \qquad 
  \left\lbrace
 \begin{array}{l}
  \mathbf{e_2}\cdot X=\xi_2 \\
  \mathbf{e_2}\cdot X=\hat{\xi}_2
  \end{array} \right. $$
  where 
 $$
 \mathbf{e_1}=\begin{pmatrix}
 \alpha \\
 \beta
 \end{pmatrix}\qquad  \mathbf{e_2}=\begin{pmatrix}
 \beta \\
 \alpha
 \end{pmatrix}
 $$
 and 
$$\alpha=\frac{1}{\sqrt{\lambda+2\mu}}-\frac{1}{\sqrt{\mu}} \quad \beta=\frac{1}{\sqrt{\lambda+2\mu}}+\frac{1}{\sqrt{\mu}}.$$
Assume that $\hat{\xi}_1-\xi_1=\hat{\xi}_2-\xi_2$.
Then an eigenvalue of the parallelogram is given by
\begin{equation}\label{eigparall}
\Lambda=\frac{2\pi^2 \sqrt{\mu(\lambda+2\mu)}}{|\Omega|}
\end{equation}
with an eigenvector $U$ of the form $U=(u,u)$ where
  \begin{equation}\label{formulau2}
 u(x,y)=\sin(\omega_1 (x+y)-\theta_1) - \sin(\omega_2 (x-y)-\theta_2)
 \end{equation}
 with
 $$\omega_1^2=\frac{\Lambda}{2\lambda+4\mu} \quad \omega_2^2=\frac{\Lambda}{2\mu}.$$
\end{theorem}
\begin{remark}
In the above synthesis, we have studied the case $C_1=C_2$. We claim that in the case $C_1\not= C_2$ there are no (bounded)
domain $\Omega$ in the plane
such that 
$$u(v,w)=C_1 \sin(\omega_1 w-\theta_1) - C_2 \sin(\omega_2 v-\theta_2)=0\quad \mbox{on the boundary }\partial\Omega.$$
Indeed if we would have two level lines of the function $u(v,w)$ crossing at some point $A$, necessarily the gradient of $u$ must vanish
at $A$. That would provide the three relations
$$\left\lbrace
\begin{array}{l}
C_1 \sin(\omega_1 w-\theta_1) - C_2 \sin(\omega_2 v-\theta_2) = 0\\
C_1 \cos(\omega_1 w-\theta_1) = 0\\
C_2 \cos(\omega_2 v-\theta_2) = 0\\
\end{array}\right.$$
that are clearly incompatible since we can assume $C_1\not= 0$ and $C_2\not= 0$ for a bounded domain.
\end{remark}
\begin{remark}
If we look for domains for which the eigenvector is $U=(u(x,y), Au(x,y))$ for some real number $A$, following the same approach we get
other parallelograms but their eigenvalue is still given by the formula \eqref{eigparall}.
\end{remark}

\medskip
Let us come back to the possible minimality of the disk. We have seen in Theorem \ref{theonotdisk} that the disk cannot be a minimizer
if the Poisson coefficient is less than $\nu^*\simeq 0.349...$ but we were not able to conclude for larger values of the Poisson coefficient
(between $\nu^*$ and $0.5$) 
since we know  that in this case the first eigenvalue of the disk is simple and the disk is a local minimizer (at least in a weak sense).
Now our previous analysis allows us to increase the interval of values of the Poisson coefficient for which the disk is not optimal:
\begin{corollary}
Assume that the Poisson coefficient $\nu$ satisfies
$$\nu< \frac{j_{1,1}^4-8\pi^2}{2(j_{1,1}^4-4\pi^2)} \simeq 0.3879$$
then the disk is not a minimizer of $\Lambda$ (among sets of given volume).
\end{corollary}
\begin{proof}
According to Theorem \ref{theonotdisk}, it suffices to compare our previous parallelogram of area $\pi$ with the first 
eigenvalue of the disk that is $\mu j_{1,1}^2$. Thus we get the thesis as soon as
$2\pi  \sqrt{\mu(\lambda+2\mu)} < \mu j_{1,1}^2$. This is equivalent to $\frac{\lambda}{\mu} +2 < \frac{j_{1,1}^4}{4\pi^2}$.
Now using the relation between the Lam\'e coefficents and the Poisson coefficient, we know that $\lambda/\mu = 2\nu /(1-2\nu)$.
Therefore
$$\frac{\lambda}{\mu} +2 < \frac{j_{1,1}^4}{4\pi^2} \Leftrightarrow 8\pi^2(1-\nu) < j_{1,1}^4 (1-2\nu) \Leftrightarrow
\nu< \frac{j_{1,1}^4-8\pi^2}{2(j_{1,1}^4-4\pi^2)}.$$
\end{proof}
\subsection{Rectangles}\label{secrectangle}
Now in the range $\frac{3}{8} \leq \nu \leq \frac{2}{5}$, that corresponds to $a=1/(1-2\nu)$ in the range $[4,5]$, we are going to consider 
convenient rectangles. Note that $3/8 <0.38$, therefore we will be able to cover the whole range $\nu \in (-1,0.4]$ and prove the disk is not
optimal in this range with these different arguments.

We consider a rectangle $\Omega_L=(0,L)\times (0,\ell)$ of area $\pi$. It will be useful to write the length and the width of the rectangle as
$$L=\sqrt{\frac{\pi}{t}}\quad\mbox{ and }\quad \ell=\sqrt{t\pi},\qquad t\in (0,1].$$
Let us denote by $\varphi_1$ the first (normalized) eigenfunction for the Dirichlet-Laplacian of $\Omega_L$, defined by
$$\varphi_1(x,y)=\frac{2}{\sqrt{\pi}}\,\sin\left(\frac{\pi x}{L}\right) \sin\left(\frac{\pi y}{\ell}\right),$$
and another eigenfunction
$$\varphi_2(x,y)=\frac{2}{\sqrt{\pi}}\,\sin\left(\frac{2\pi x}{L}\right) \sin\left(\frac{2\pi y}{\ell}\right).$$
This other eigenfunction could be the fourth one (for a rectangle not too far from the square), but can also have a larger index.
We will explain below why we do this particular choice.

Now the idea is to plug in the variational formulation defining $\Lambda(\Omega_L)$ a family of vectors, for 
$X=(\alpha_1,\alpha_2,\beta_1,\beta_2)$:
$$U_X=\left(\begin{array}{c}
u_1\\
u_2
\end{array}\right)=\left(\begin{array}{c}
\alpha_1 \varphi_1 + \alpha_2 \varphi_2\\
\beta_1 \varphi_1 + \beta_2 \varphi_2\\
\end{array}\right).$$
Since the eigenfunctions of the Laplace operator define an orthonormal basis, we have
$$\int_{\Omega_L}  |\nabla u_1|^2 +|\nabla u_2|^2 =\left(\frac{\pi^2}{L^2}+\frac{\pi^2}{\ell^2}\right)\left(\alpha_1^2+4\alpha_2^2+\beta_1^2+4\beta_2^2 \right)$$
and 
$$\int_{\Omega_L} u_1^2 + u_2^2 =\alpha_1^2+\alpha_2^2+\beta_1^2+\beta_2^2.$$
It remains to compute $\int_{\Omega_L} (\mathrm{div}(U_X))^2$. We obtain
\begin{equation}
\int_{\Omega_L} (\mathrm{div}(U_X))^2=\frac{\pi^2}{L^2}\left(\alpha_1^2+4\alpha_2^2\right) + \frac{\pi^2}{\ell^2}\left(\beta_1^2+4\beta_2^2\right)
- \frac{128}{9\pi}\left(\alpha_1 \beta_2 + \alpha_2 \beta_1\right).
\end{equation}
Using $a=1/(1-2\nu)$, this implies using this admissible test function, that
\begin{equation}\label{ratioQ}
\frac{\Lambda(\Omega_L)}{\mu} \leq \frac{Q(X)}{\alpha_1^2+\alpha_2^2+\beta_1^2+\beta_2^2}
\end{equation}
where $Q$ is the quadratic form defined by
\begin{eqnarray*}
Q(X)=\alpha_1^2\left((1+a)\frac{\pi^2}{L^2} + \frac{\pi^2}{\ell^2}\right) + \alpha_2^2\left(4(1+a)\frac{\pi^2}{L^2} + \frac{4\pi^2}{\ell^2}\right) +
\beta_1^2\left(\frac{\pi^2}{L^2} + (1+a)\frac{\pi^2}{\ell^2}\right) + \\
 \beta_2^2\left(\frac{4\pi^2}{L^2} + 4(1+a)\frac{\pi^2}{\ell^2}\right) - \frac{128 a}{9\pi}\left(\alpha_1 \beta_2 + \alpha_2 \beta_1\right).
\end{eqnarray*}
Now we have to choose $X=(\alpha_1,\alpha_2,\beta_1,\beta_2)$ that give the lowest possible value for the ratio in \eqref{ratioQ}.
This lowest value exactly corresponds to the smallest eigenvalue of the $4\times 4$ matrix of the quadratic form $Q$.
This matrix $\mathcal{M}$ has the simple structure
$$\mathcal{M}=\left(\begin{array}{cccc}
a_1 & 0 & 0 & b\\
0 & a_2 & b & 0\\
0 & b & a_3 & 0\\
b & 0 & 0 & a_4\\
\end{array}\right) .$$
Its characteristic polynomial factorizes as
$P(x)=[(a_2-x)(a_3-x)-b^2][(a_1-x)(a_4-x)-b^2]$ with $b=-64a/9\pi$ and
$$\begin{array}{l}
a_1=\pi(1+a)t+\frac{\pi}{t }\\
a_2=4\pi(1+a)t+\frac{4\pi}{t}\\
a_3=\pi t+\frac{(1+a)\pi}{t }\\
a_4=4\pi t+\frac{4(1+a)\pi}{t } .\\
\end{array}$$
We observe that $a_2a_3=a_1a_4$ and $a_1+a_4 \geq a_2+a_3$ because $t\leq 1$. Therefore the trinome $(a_1-x)(a_4-x)-b^2$ is always less than
$(a_2-x)(a_3-x)-b^2$ and the smallest root of $P(x)$ is the smallest root of $q_1(a,t,x):=(a_1-x)(a_4-x)-b^2$.
More precisely, the question is to know whether the smallest root of $q_1$ is smaller than $j_{1,1}^2$ because our aim is to compare
the rectangle $\Omega_L$ with the unit disk.
Since $q_1(a,t,0)=a_1a_4-b^2$, we see that
$$q_1(a,t,0)\geq q_1(a,\frac{1}{2},0) = 4\pi^2\left(a^2+4a+4-\frac{1024 a^2}{81\pi^4}\right) >0.$$
Therefore, we get the thesis as soon as we can find some $t^*\in (0,1]$ such that $q_1(a,t^*,j_{1,1}^2)<0$ for all $a\in [4,5]$.
It turns out that the particular choice $t^*=2/5=0.4$ achieves this aim. This is an elementary analysis to prove that the polynomial expression
$$
q_1(a,\frac{2}{5},j_{1,1}^2)=j_{1,1}^4 - \pi j_{1,1}^2 \left(\frac{29}{2} + \frac{52 a}{5}\right) 
+ 4\pi^2 \left(a^2+\frac{169}{36}\,(a+1)\right) - \frac{4096 a^2}{81\pi^2 }
$$
remains negative for all $a\in [4,5]$. Thus we have proved
\begin{theorem}
Let $\Omega_L$ be the rectangle of length $L=\sqrt{5\pi/2}$ and width $\ell=\sqrt{2\pi/5}$. Then its first eigenvalue satisfies
$$\Lambda(\Omega_L) < \mu j_{1,1}^2$$
for all values of the Poisson coefficient $\nu \in [\frac{3}{8},\frac{2}{5}]$.
Therefore the disk is not a minimizer in this range of values of $\nu$.
\end{theorem}
\begin{remark}
Let us explain why we chose the association of $\varphi_1$ and $\varphi_2$ as test functions. The aim is to get a cross product
coming from the divergence term strong enough to make the first eigenvalue of the matrix $\mathcal{M}$ as small as possible.
It turns out that a choice of the two first eigenfunctions of the rectangle would not realize this and a simple analysis convince us
that our choice was the better.
\end{remark}
\subsection{The case of ellipses}
In the case where the domain $\Omega$ is an ellipse, we do not have an explicit expression for the eigenvalues 
(nor a good upper estimate)  and eigenfunctions,
as we did in Sections~\ref{secrhombi} and \ref{secrectangle}. Therefore, we have aimed to extend the previous analysis using numerical simulations, in which we computed an estimate of the eigenvalue $\Lambda(\Omega_a)$, where $\Omega_a$ denotes the ellipse with semi-axes $a$ and $1/a$. For \( a = 1 \), this corresponds to the eigenvalue of the disk, equal to \( \mu j_{1,1}^2 \) as long as \( \nu \geq 0.35 \), as stated in Theorem~\ref{theopoisson}. Numerical observations summarized on Figure~\ref{fig:caseEllipses}, performed with the software Matlab, suggest that the disk is not optimal while \( \nu \leq \bar{\nu} \), and it appears to be optimal among  ellipses with area $\pi$ when \( \nu > \bar{\nu} \) where $\bar{\nu}$ is a numerical value
very close to $0.41$.

\begin{figure}[htbp]
    \centering
    \subfigure[Case $\nu=0.39$]{\includegraphics[width=0.32\textwidth]{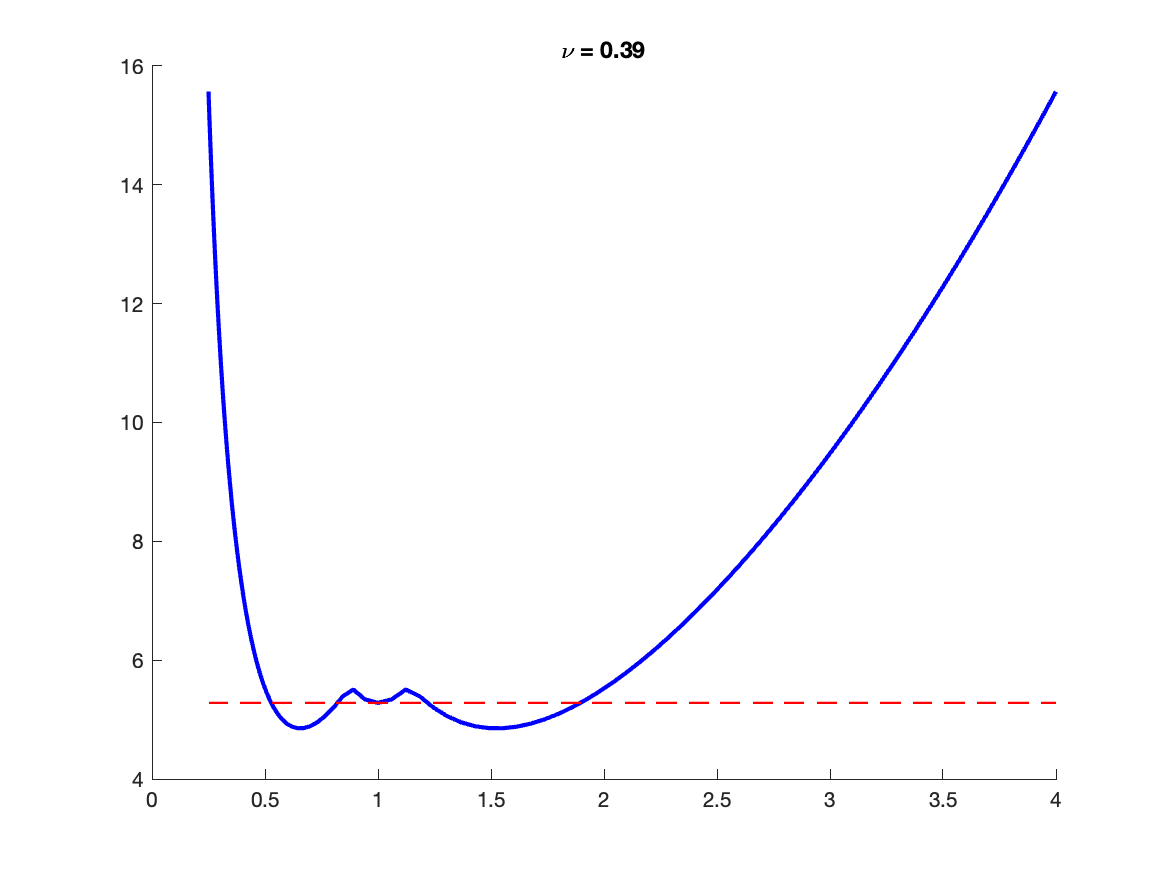}}
    \subfigure[Case $\nu=0.40$]{\includegraphics[width=0.32\textwidth]{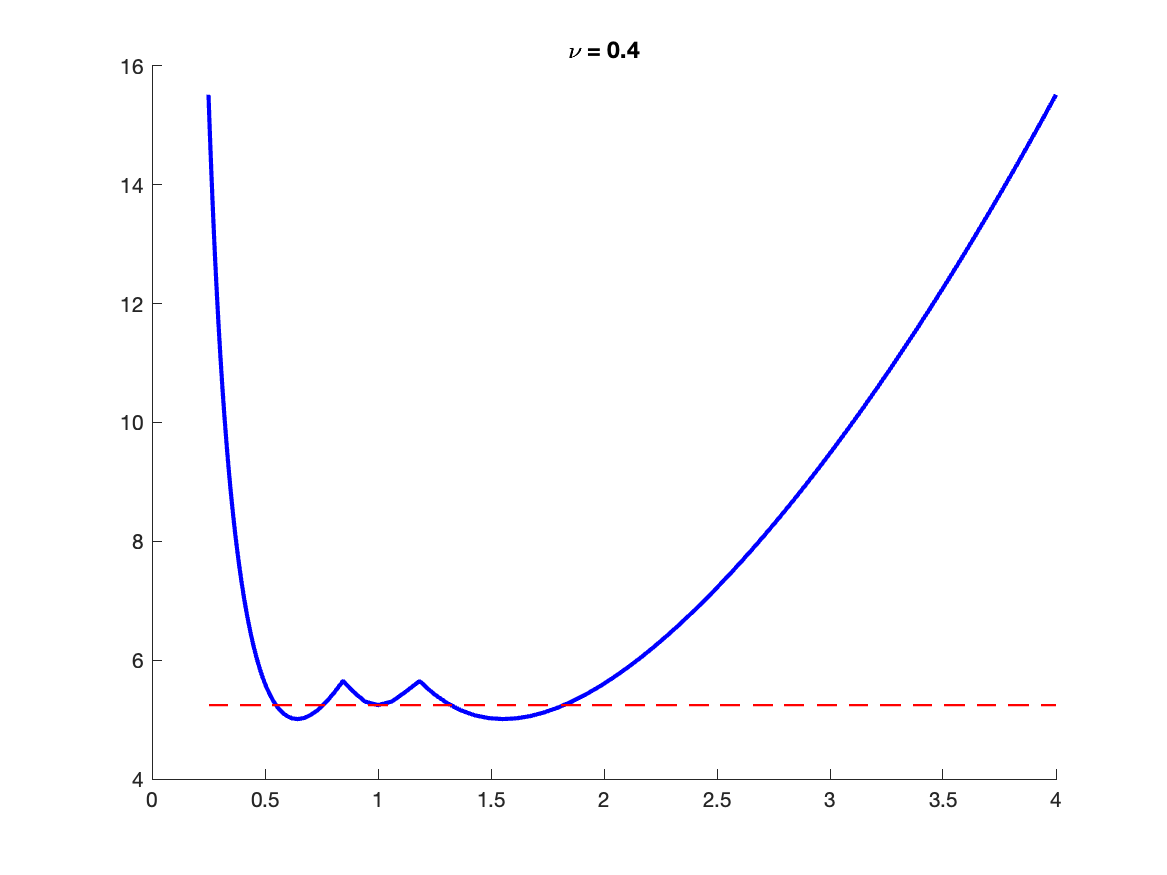}}
    \subfigure[Case $\nu=0.405$]{\includegraphics[width=0.32\textwidth]{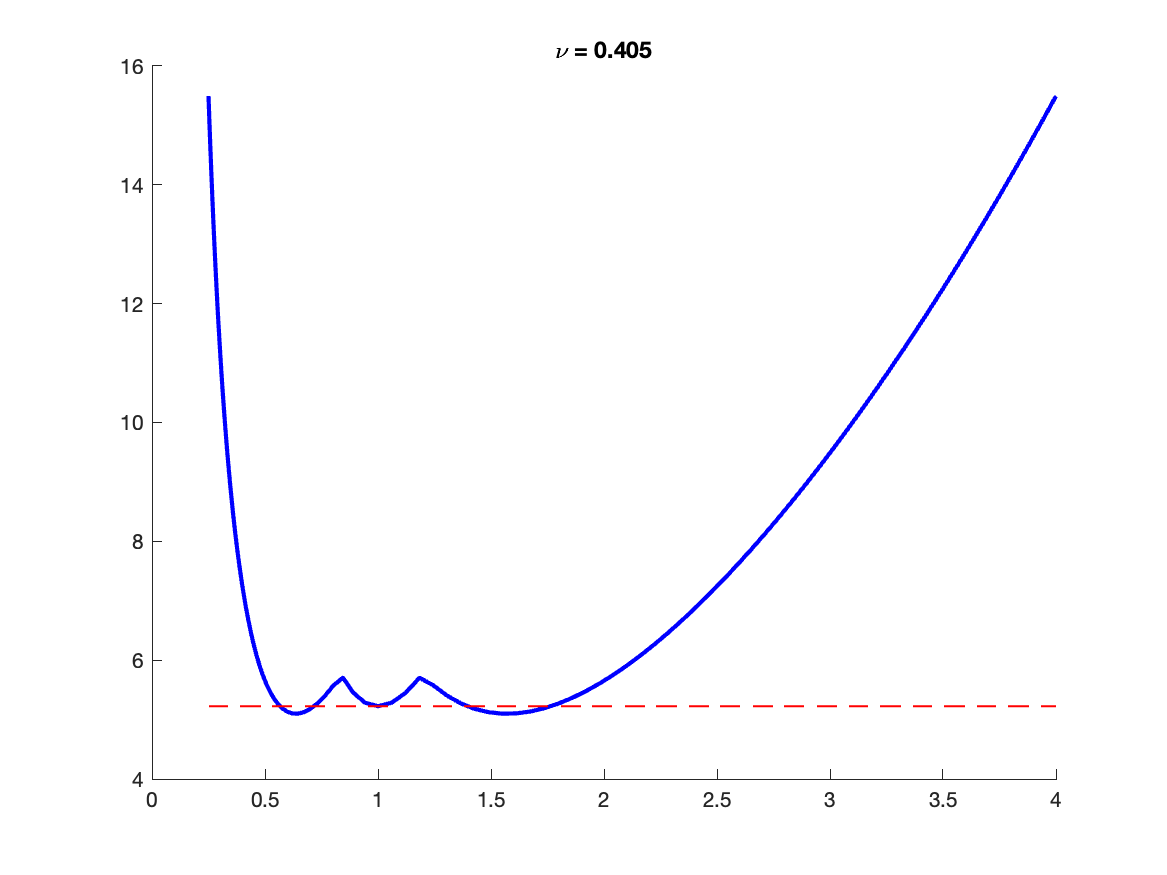}}\\
    \subfigure[Case $\nu=0.41$]{\includegraphics[width=0.32\textwidth]{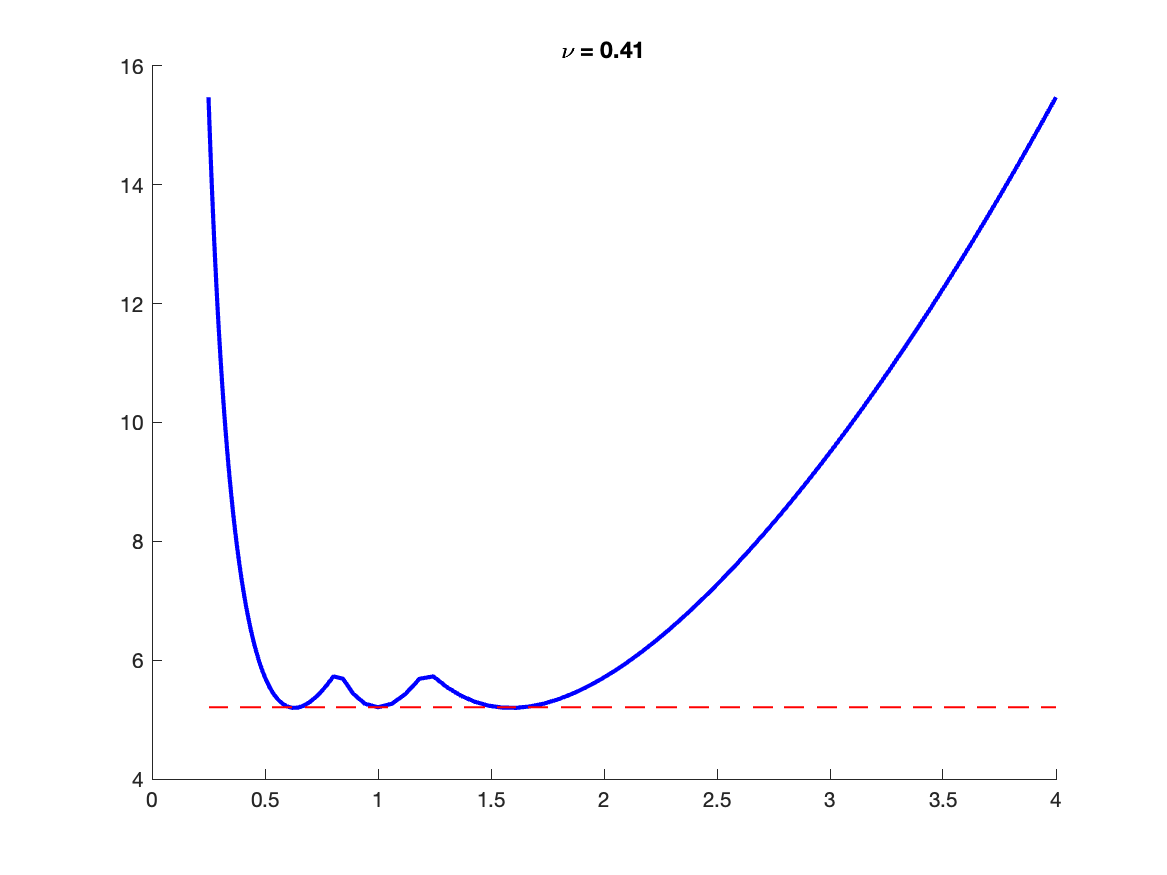}}
    \subfigure[Case $\nu=0.42$]{\includegraphics[width=0.32\textwidth]{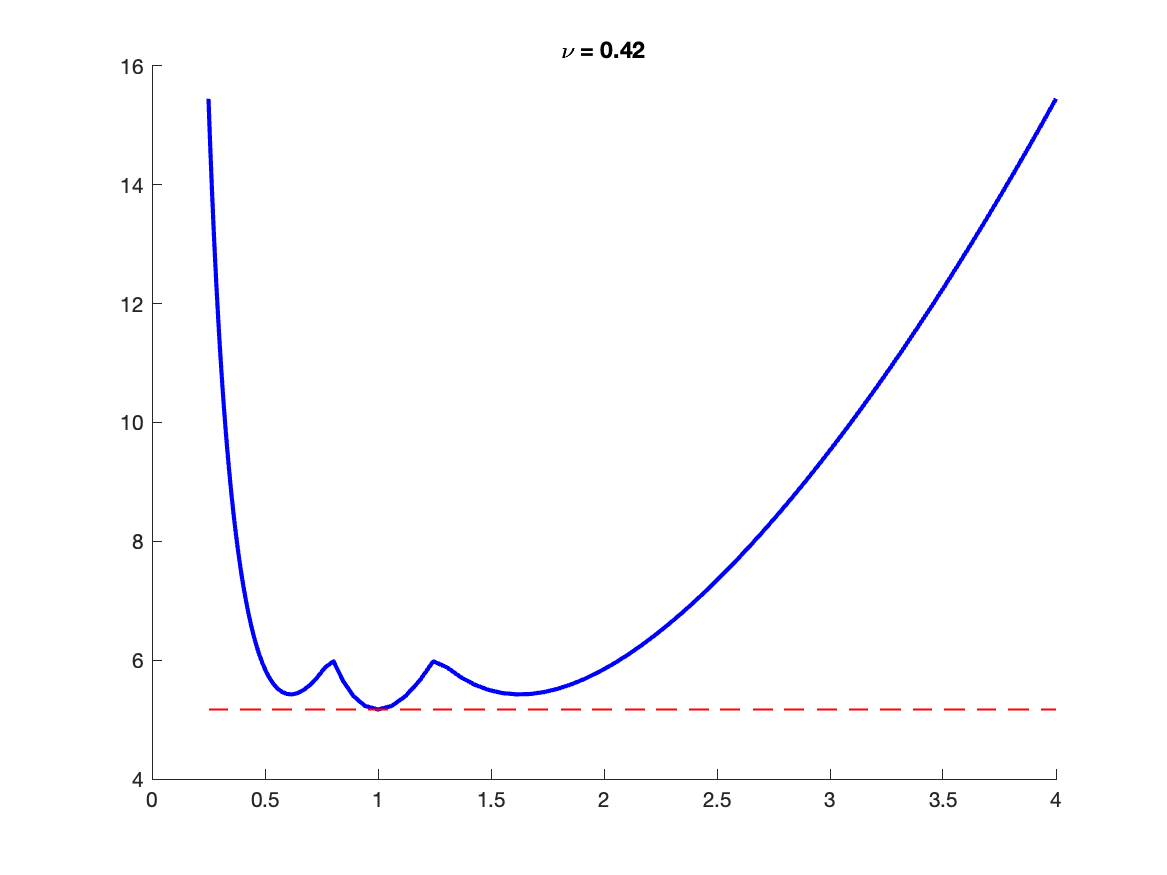}}
    \subfigure[Case $\nu=0.45$]{\includegraphics[width=0.32\textwidth]{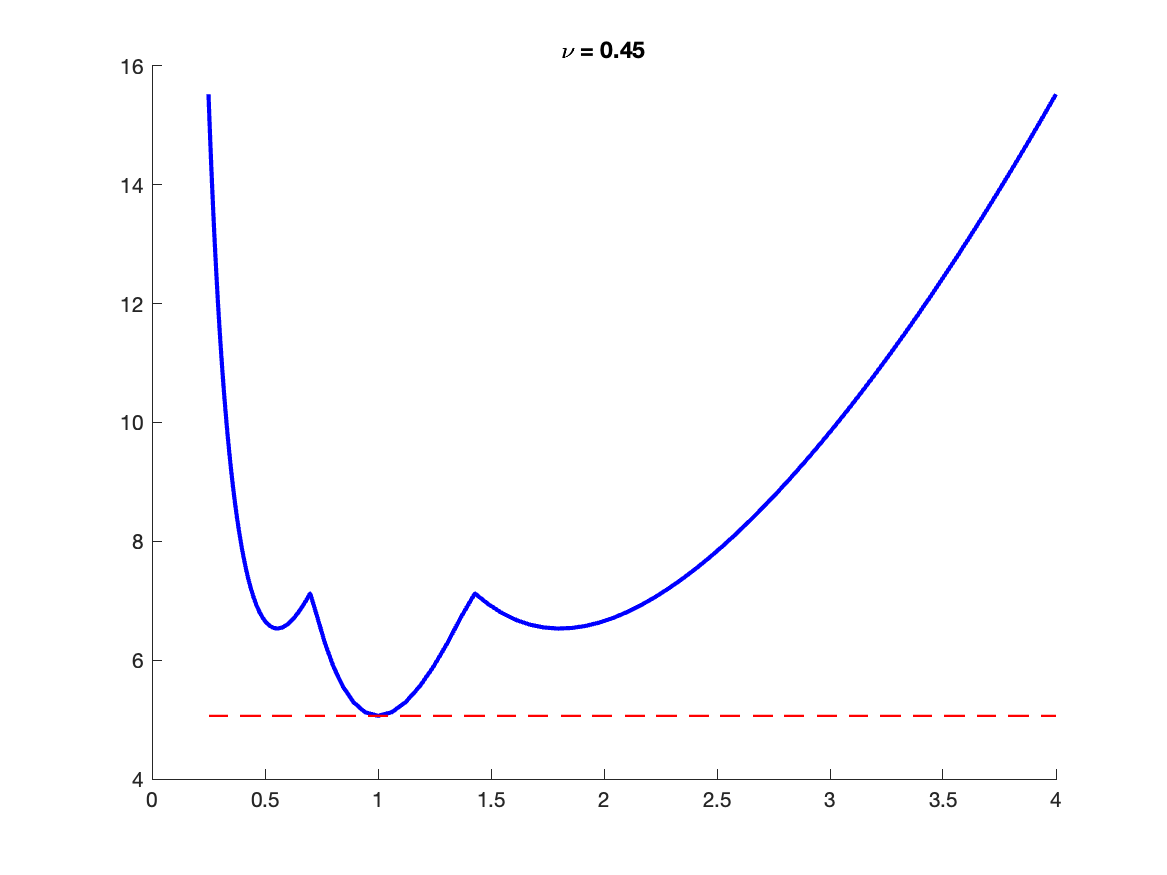}}
    \caption{Graph of $\Lambda(\Omega_a)$ with respect to $a$. The dotted line corresponds to the first eigenvalue of the disk.}
    \label{fig:caseEllipses}
\end{figure}

\section{Conclusion}\label{secconclusion}
\subsection{A conjecture}
In our numerical simulations, we are not able to give a lower first eigenvalue than the disk when $\nu \geq 0.41$.
For example, the best rectangles are better than the disk only when $\nu$ is less than a value not far from $0.41$. It is the same
for the best ellipse. This leads us to think that it may exist a threshold value $\hat{\nu}$ such that the disk becomes the 
solution of our minimization problem when $\nu \geq \hat{\nu}$. An heuristic argument that supports this conjecture is the following:
\begin{enumerate}
\item First we prove in Section~\ref{secgamma} below that the eigenvalues of the Lam\'e system converge to the eigenvalues of the Stokes system when
$\nu \to 1/2$.
\item If we assume that the disk minimizes the first Stokes eigenvalue in the plane (this is another conjecture as explained in \cite{henrot-mazari-privat}),
\item if we could then use the local minimality of the disk for our problem in a strong sense (for example for the Hausdorff convergence),
\end{enumerate}
we would get the result. Indeed, by the $\Gamma$-convergence result stated below, we see that the minimizer for the Lam\'e system must converge
to the minimizer for the Stokes system. Therefore, for $\nu$ close enough to $1/2$, the minimizer should enter in the neighborhood of the disk
where the disk is the solution. This could provide the expected result.
 \subsection{$\Gamma$-convergence as $\nu \to 1/2$}\label{secgamma}
As explained just above, it is interesting to prove that when $\nu\to 1/2$,    the eigenvalues of the Lam\'e
system converge to those of the Stokes system.

 For that purpose we will renormalize the eigenvalue  $\Lambda(\Omega)$ and work with the parameter $a:=\frac{\lambda+\mu}{\mu}=\frac{1}{1-2\nu}$  that we consider satisfying $a\to+\infty$. In other words  the right quantity to study becomes
 
 $$\Lambda^a(\Omega):=\frac{1}{\mu}\Lambda(\Omega)=\min_{u\in H^1_0(\Omega) \setminus \{0\}} \frac{\int_{\Omega}|\nabla u|^2 \;dx + a \int_{\Omega} ({\rm div}(u))^2\;dx}{\int_{\Omega} |u|^2 \;dx}.$$

 In this section we would like to investigate the limiting behavior as $a \to +\infty$ (or equivalently $\nu\to 1/2$) of $\Lambda^a(\Omega)$.  In particular we will show that for $\Omega$ fixed,  $\Lambda^{a}(\Omega)$ converges to the Stokes eigenvalue, and moreover under some standard geometrical restrictions on the admissible sets $\Omega$,  the shape functional $\Omega \mapsto \Lambda^a(\Omega)$ $\Gamma$-converges to $\Omega \mapsto \lambda_1^{\rm Stokes}(\Omega)$. Here,  $\lambda_1^{\rm Stokes}(\Omega)$ is the Stokes eigenvalue already defined in Section \ref{remark1} as 
  $$\lambda_1^{\rm Stokes}(\Omega):=\min_{u\in H^1_0(\Omega) \setminus \{0\} \text{ s.t. } {\rm div}(u)=0} \frac{\int_{\Omega}|\nabla u|^2 \;dx }{\int_{\Omega} |u|^2 \;dx}.
$$
 
We first establish the convergence of $\Lambda^a(\Omega)$ for $\Omega$, fixed. For this purpose we define the following two quadratic forms on $H^1_0(\Omega)$ :
$$
Q_a(u):=  \int_{\Omega}|\nabla u|^2 \;dx + a\int_{\Omega} ({\rm div}(u))^2\;dx.
$$

$$
Q_{\infty}(u):= 
\left\{
\begin{array}{cc}
 \int_{\Omega}|\nabla u|^2 \;dx  & \text{ if } {\rm div}(u)=0\\ 
+\infty  & \text{ otherwise}.
\end{array}\right.
$$

\begin{proposition}\label{gamma1} Let $\Omega$  be a bounded open set. Then $Q_a$ $\Gamma$-converges to $Q_\infty$ for 
the $L^2$ topology. when $a\to +\infty$  As a consequence, the associated Lam\'e operator converges in the strong resolvent sense to the Stokes operator and in particular
$$\lim_{a \to \infty} \Lambda^a (\Omega)= \lambda_1^{\rm Stokes}(\Omega).$$
\end{proposition} 

\begin{proof} The proof is somehow standard. Let us write the details. \\

\emph{Step 1. $\Gamma$-limsup}. Let  $u \in H^1_0(\Omega)$  be such that $Q_\infty(u)<+\infty$ (otherwise there is nothing to prove). Then we take as a recovery sequence the constant sequence $u_\lambda=u$ and  we use that  ${\rm div}(u)=0$, together with Korn inequality to  deduce that
$$Q_a(u)=Q_{\infty}(u)$$
and a fortiori,
$$\limsup_{a \to +\infty} Q_a(u)=Q_{\infty}(u),$$
which directly proves the limsup inequality.\\

\emph{Step 2. $\Gamma$-liminf}.  Assume that $u_a \to u$ in $L^2(\Omega)$.  We can assume that 
$$\sup_{a} Q_a(u_a) \leq C,$$
otherwise there is nothing to prove. But this means thanks to Korn inequality, that $u_a$ is uniformly bounded in $H^1(\Omega)$ thus converges weakly in $H^1$ and strongly in $L^2$, up to a subsequence, to some function $u \in H^1_0(\Omega)$ and
$$\int_{\Omega} |\nabla u|^2   \leq \liminf_{a} \int_{\Omega} | \nabla u_a|^2 \;dx.$$

 Passing to the liminf in the inequality
$$ \int_{\Omega} ({\rm div}(u))^2\;dx\leq \frac{C}{a},$$
we deduce that ${\rm div}(u )= 0$ thus   
$$Q_{\infty}(u)= \int_{\Omega} |\nabla u|^2 dx $$
which finishes the liminf inequality, and the proof of $\Gamma$-convergence.

Then the end of the statement of the Proposition follows from the standard theory of $\Gamma$-convergence that asserts that $\Gamma$-convergence of quadratic forms implies the convergence in the strong resolvent sense of the associated operators (see \cite[Chapter 12]{dalmasoGamma}). A review of these properties can also be found in \cite[Section 1.1]{PabloAntoine}.  In particular, the convergence of the eigenvalues  follows from the fact that the associated operators have compact resolvent. More precisely, we first notice that the quadratic forms $Q_a$ and $Q_\infty$ are equi-coercive thanks to Poincar\'e-Korn inequality. They are also semi-continuous with respect to the $L^2$ topology. The associated operators are thus self-adjoint, invertible and thanks to the compact embedding of $H^1_0(\Omega)$ into $L^2(\Omega)$, their inverse are compact operators.   We then apply Proposition 7 in \cite{PabloAntoine} with $X=H^1_0(\Omega)$ and $H=L^2(\Omega)$ which establishes the convergence of the spectrum for the inverse operators, from the $\Gamma$-convergence of $Q_a$ to $Q_\infty$. The spectrum of the operators itself then follows immediately.
\end{proof}

Now we consider the $\Gamma$-convergence of $\Lambda^a(\Omega)$ but  with respect to the variable  $\Omega$. For simplicity we will work in the restricted class of domains $\Omega$ that are uniformly Lipschitz, more precisely that satisfies a uniform $\varepsilon$-cone property (see \cite[Definition 2.4.1]{zbMATH06838450}). We will endow this class with the complementary Hausdorff distance (see \cite[Definition 2.2.8]{zbMATH06838450}) and we already know that the Dirichlet problem is stable along any such converging sequence in this class (\cite[Theorem 3.2.13]{zbMATH06838450}) which will help a lot in the following proposition.

\begin{proposition} Let $D\subset \R^N$, $\varepsilon_0>0$ and $V>0$, fixed, and let $\mathcal{A}_0$ be the class of domains $\Omega \subset D$ that satisfy the $\varepsilon_0$-cone property  together with the further constraint $|\Omega|=V$. Then as $a\to +\infty$, the family of functional $\Lambda^a :\mathcal{A} \to \R$, $\Gamma$-converges to $\lambda_1^{\rm Stokes}$ with respect to the  complementary Hausdorff distance. 
\end{proposition}

\begin{proof} 
\emph{Step 1. $\Gamma$-limsup}. For $\Omega \in \mathcal{A}$ being given we take the constant sequence $\Omega_a=\Omega$ as a recovery sequence. Thanks to Proposition \ref{gamma1} we know that 
$$\lim_{a\to +\infty} \Lambda^a (\Omega)= \lambda_1^{\rm Stokes}(\Omega),$$
which proves the $\Gamma$-limsup property.\\

\emph{Step 2. $\Gamma$-liminf}.   Let $\Omega_a$ converging to $\Omega$ for the complementary Hausdorff distance. Since $\mathcal{A}$ is closed for the complementary Hausdorff convergence (see \cite[Theorem 2.4.10]{zbMATH06838450}), it follows that  $\Omega \in \mathcal{A}$, and it is easily seen that $|\Omega|=V$. 

Let  $u_a$ be a sequence of normalized eigenfunctions, associated to $\Lambda^a(\Omega_a)$. In particular, $u_a \in H^1_0(\Omega_a)$ and 
$$\Lambda^a(\Omega_a)= \int_{\Omega_\lambda}|\nabla u_a|^2 \;dx + a \int_{\Omega_a} ({\rm div}(u_a))^2\;dx.$$
We may assume without loss of generality that $(\Lambda^a(\Omega_a))_{a}$ is a bounded sequence. Therefore,   the sequence $u_a$ is uniformly bounded in $H^1_0(D)$ and converges up to a subsequence (not relabelled) to a function $u \in H^1_0(D)$, weakly in $H^1(D)$ and strongly in $L^2(D)$. In particular $\|u\|_{L^2(D)}=1$. Moreover by the Mosco convergence  of $\Omega_a$ (see \cite{zbMATH06838450}) we know that $u \in H^1_0(\Omega)$.  Finally, the bound on $\Lambda^a(\Omega_a)$ tells us that 
$$\int_{\Omega_a}({\rm div}(u_a))^2 \;dx \leq \frac{C}{a}\to_{a\to +\infty} 0,$$
and we deduce from the weak convergence of $u_a$ to $u$ in $H^1(D)$ that  ${\rm div}(u)=0$ thus $u$ is an admissible competitor for the Rayleigh quotient that defines $\lambda_1^{\rm Stokes}(\Omega)$.

Therefore, the following sequence of inequalities holds:
\begin{eqnarray}
\lambda_1^{\rm Stokes}(\Omega)&\leq &  \int_{\Omega }|\nabla u|^2 \;dx \notag \\
&\leq & \liminf  \int_{D}|\nabla u_a|^2 \;dx  \notag \\
&\leq &  \liminf  \int_{D}|\nabla u_a|^2 \;dx + a \int_{D} ({\rm div}(u_a))^2\;dx \notag \\
&= & \liminf  \Lambda^a(\Omega_a), \notag 
\end{eqnarray}
which finishes the proof of the liminf inequality, and so follows the proof of $\Gamma$-convergence.
\end{proof}

 \section*{Acknowledgments} This work  is partially supported by the ANR project STOIQUES financed by the French Agence Nationale de la Recherche (ANR).    We would like to warmly thank Michael Levitin  for valuable discussions and also David Krejcirik and Davide Buoso
 for interesting feedback on a preliminary version of the paper.

\bibliographystyle{abbrv}
\bibliography{biblio_korn}

 \end{document}